\documentclass[a4paper,11pt]{article}
\usepackage[utf8x]{inputenc}

\usepackage{a4wide}
\usepackage{amssymb}
\usepackage{amsmath}
\usepackage{amsthm}
\usepackage[mathscr]{euscript}
\usepackage[normalem]{ulem}
\usepackage[hidelinks]{hyperref}
\usepackage[usenames,dvipsnames]{xcolor}
\usepackage{todonotes}
\usepackage{appendix}
\usetikzlibrary{arrows,angles}
\usetikzlibrary{matrix}
\usetikzlibrary{snakes}
\usetikzlibrary{arrows,calc,shapes,decorations.pathreplacing}
\usepackage{tikz-cd}
\usepackage{physics}
\usepackage{comment}
\usepackage{bm}

\numberwithin{equation}{section}

\theoremstyle{definition}

\newtheorem{Definition}{Definition}[section]

\newtheorem{Remark}[Definition]{Remark}
\newtheorem{Example}[Definition]{Example}

\theoremstyle{plain}

\newtheorem{Theorem}[Definition]{Theorem}
\newtheorem{Proposition}[Definition]{Proposition}
\newtheorem{Lemma}[Definition]{Lemma}
\newtheorem{Corollary}[Definition]{Corollary}

\newcommand{\R}{\mathbb R}
\newcommand{\Z}{\mathbb Z}
\newcommand{\N}{\mathbb N}
\newcommand{\Fl}{\text{\normalfont Fl}}
\newcommand{\Vol}{\mathrm{Vol}}

\newcommand{\eps}{\varepsilon}
\newcommand{\vphi}{\varphi}
\usepackage[xcolor]{changebar}
\cbcolor{blue}
\newcommand{\Ric}{\mathrm{Ric}}
\newcommand{\comp}{\Subset}

\newcommand{\D}{\mathcal{D}}

\newcommand{\Riem}{\mathrm{Riem}}
\newcommand{\Test}{\mathrm{Test}}
\newcommand{\TestV}{\mathrm{TestV}}
\newcommand{\TestF}{\mathrm{TestF}}

\newcommand{\cT}{T}

\usepackage[inline]{enumitem}
\newcommand{\enumlabelformat}{\roman}
\newcommand{\enumlabelfont}[1]{#1}
\newlength{\thelabelsep}
\setlist{labelsep=\thelabelsep}
\setlist[enumerate,1]{font=\enumlabelfont,label=(\enumlabelformat*),leftmargin=2.5em}
\setlist[itemize]{leftmargin=2.5em,label=$-$}

\newcounter{inlineenum}
\renewcommand{\theinlineenum}{\enumlabelformat{inlineenum}}
\let\epsilon\varepsilon
\let\phi\varphi

\makeatletter
\let\save@mathaccent\mathaccent
\newcommand*\if@single[3]{%
  \setbox0\hbox{${\mathaccent"0362{#1}}^H$}%
  \setbox2\hbox{${\mathaccent"0362{\kern0pt#1}}^H$}%
  \ifdim\ht0=\ht2 #3\else #2\fi
  }
\newcommand*\rel@kern[1]{\kern#1\dimexpr\macc@kerna}
\newcommand*\widebar[1]{\@ifnextchar^{{\wide@bar{#1}{0}}}{\wide@bar{#1}{1}}}
\newcommand*\wide@bar[2]{\if@single{#1}{\wide@bar@{#1}{#2}{1}}{\wide@bar@{#1}{#2}{2}}}
\newcommand*\wide@bar@[3]{%
  \begingroup
  \def\mathaccent##1##2{%
    \let\mathaccent\save@mathaccent
    \if#32 \let\macc@nucleus\first@char \fi
    \setbox\z@\hbox{$\macc@style{\macc@nucleus}_{}$}%
    \setbox\tw@\hbox{$\macc@style{\macc@nucleus}{}_{}$}%
    \dimen@\wd\tw@
    \advance\dimen@-\wd\z@
    \divide\dimen@ 3
    \@tempdima\wd\tw@
    \advance\@tempdima-\scriptspace
    \divide\@tempdima 10
    \advance\dimen@-\@tempdima
    \ifdim\dimen@>\z@ \dimen@0pt\fi
    \rel@kern{0.6}\kern-\dimen@
    \if#31
      \overline{\rel@kern{-0.6}\kern\dimen@\macc@nucleus\rel@kern{0.4}\kern\dimen@}%
      \advance\dimen@0.4\dimexpr\macc@kerna
      \let\final@kern#2%
      \ifdim\dimen@<\z@ \let\final@kern1\fi
      \if\final@kern1 \kern-\dimen@\fi
    \else
      \overline{\rel@kern{-0.6}\kern\dimen@#1}%
    \fi
  }%
  \macc@depth\@ne
  \let\math@bgroup\@empty \let\math@egroup\macc@set@skewchar
  \mathsurround\z@ \frozen@everymath{\mathgroup\macc@group\relax}%
  \macc@set@skewchar\relax
  \let\mathaccentV\macc@nested@a
  \if#31
    \macc@nested@a\relax111{#1}%
  \else
    \def\gobble@till@marker##1\endmarker{}%
    \futurelet\first@char\gobble@till@marker#1\endmarker
    \ifcat\noexpand\first@char A\else
      \def\first@char{}%
    \fi
    \macc@nested@a\relax111{\first@char}%
  \fi
  \endgroup
}
\makeatother
\newcommand*{\de}{\text{\normalfont d}}
\newcommand*{\meas}{\mathfrak{m}}

\DeclareMathOperator{\supp}{supp}
\newcommand*{\loc}{\text{\normalfont loc}}
\DeclareMathOperator{\Hess}{Hess}
\DeclareMathOperator{\Div}{div}

\title{Ricci curvature bounds and rigidity for non-smooth Riemannian and semi-Riemannian metrics}

\author{Michael Kunzinger\footnote{Department of Mathematics, University of Vienna, Oskar-Morgenstern-Platz 1, 1090 Wien, Austria. \newline michael.kunzinger@univie.ac.at \newline argam.ohanyan@univie.ac.at \newline alessio@vardabasso@univie.ac.at}\\Argam Ohanyan$^*$\\Alessio Vardabasso$^*$
}
\begin{document}

\date{\today}

%\date{Received: date /Accepted: date}

\maketitle

\begin{abstract}
We study rigidity problems for Riemannian and semi-Riemannian manifolds with metrics of low regularity. Specifically, we prove a version of the Cheeger-Gromoll splitting theorem \cite{CheegerGromoll72splitting} for Riemannian metrics and the flatness criterion for semi-Riemannian metrics of regularity $C^1$. With our proof of the splitting theorem, we are able to obtain an isometry of higher regularity than the Lipschitz regularity guaranteed by the $\mathsf{RCD}$-splitting theorem \cite{gigli2013splitting, gigli2014splitoverview}. Along the way, we establish a Bochner-Weitzenböck identity which permits 
both the non-smoothness of the metric and of the vector fields, complementing a recent similar result in \cite{mondino2024equivalence}. The last section of the article is dedicated to the discussion of various notions of Sobolev spaces in low regularity, as well as an alternative proof of the equivalence (see \cite{mondino2024equivalence}) between distributional Ricci curvature bounds and $\mathsf{RCD}$-type bounds, using in part the stability of the variable $\mathsf{CD}$-condition under suitable limits \cite{ketterer2017variableCD}.

\vspace{1em}

\medskip


\end{abstract}
\newpage

\tableofcontents
\newpage

\section{Introduction}\label{Section: Introduction}

The study of metrics whose regularity is below $C^{\infty}$ (in fact, below $C^2$, since for most global comparison-geometric purposes the metric is almost never differentiated more than twice) is of fundamental importance in (semi-)Riemannian geometry. Instances in which non-regular Riemannian metrics arise include Ricci flows and gluing constructions. On the semi-Riemannian side, specifically the physically relevant Lorentzian signature $(-,+,\dots,+)$, which provides the geometric framework for General Relativity (i.e., Einstein's theory of gravity), gives rise to non-smooth metrics naturally. In this context, non-smooth Lorentzian metrics need to be studied in order to understand and give physical meaning to phenomena of crucial importance in physics such as spacetime singularities. We refer to \cite{graf2020singularity, KOSS22C1HawkingPenrose, CGHKS24} and the references therein. The methods developed in these works are also of relevance to us in this article, despite the different signature, since most of the approximation and convolution based tools developed there to deal with low regularity metrics and curvature tensors are analytical results on which the signature has no bearing.

While the analysis of low regularity metrics is one way to deal with non-smooth geometry, there is also the theory of ($\mathsf{R}$)$\mathsf{CD}$-spaces, a rich field studying the geometry of (infinitesimally Hilbertian) metric measure spaces with synthetic Ricci curvature bounded below. Spaces with synthetic Ricci curvature bounds (i.e.\ $\mathsf{CD}$-spaces) are defined via entropic convexity along Wasserstein geodesics and were developed in the works of McCann, Sturm and Lott-Villani \cite{McCann97, Stu:06a, Stu:06b, LottVillani09}. The $\mathsf{CD}$-condition turns out to be too general for many purposes, since it does not distinguish between properly Riemannian and merely Riemann-Finslerian spaces. In particular, many key rigidity results such as the Cheeger-Gromoll splitting theorem \cite{CheegerGromoll72splitting} cannot be generalized to $\mathsf{CD}$-spaces. To single out strictly Riemannian spaces, Gigli 
introduced the notion of \emph{infinitesimal Hilbertianity} in \cite{gigli2015}, leading to the rich theory of \emph{Riemannian} $\mathsf{CD}$-spaces (or, $\mathsf{RCD}$-spaces, for short). Similar developments have been achieved recently in Lorentzian signature, see \cite{mccann2020displacement, mondino2022optimal, cavallettimondino2020optimal, braun2023renyi, mccann2024synthetic, ketterer2024characterization, braunmccann2023variablecurvature, Octet24+, braun2024nonsmooth, cavalletti2024optimal}. The natural question of compatibility between distributional Ricci curvature bounds for low regularity Riemannian metrics and the $\mathsf{RCD}$-condition was recently answered to the positive by Mondino and Ryborz in \cite{mondino2024equivalence} (after an initial contribution by Kunzinger, Oberguggenberger and Vickers \cite{KOV:22}). Moreover, it turns out that a splitting theorem in the spirit of Cheeger-Gromoll can be proven for $\mathsf{RCD}$-spaces, and this was achieved by Gigli \cite{gigli2013splitting, gigli2014splitoverview}.

The principal aim of this article is to establish rigidity theorems for metrics of low regularity. First, we prove the following version of the Cheeger-Gromoll splitting theorem, which has the advantage of yielding a splitting isometry of higher regularity than just the Lipschitz continuity which would follow by the $\mathsf{RCD}$-splitting theorem:

\begin{Theorem}[Cheeger-Gromoll splitting for low regularity Riemannian metrics]
Let $M$ be a smooth manifold and $g$ a Riemannian metric tensor on $M$ of regularity $C^k$, $k \geq 1$. Suppose that the following hold:
\begin{enumerate}
    \item $(M,d_g)$ is a complete metric space, where $d_g$ is the Riemannian distance on $M$ induced by $g$.
    \item $\Ric(X,X) \geq 0$ in $\mathcal{D}'(M)$ for every compactly supported smooth vector field $X$ on $M$.
    \item There exists a geodesic $c:\R \to M$ with the property $d_g(c(s),c(t)) = |s-t|$ for all $s,t \in \R$ (i.e.\ $c$ is a \emph{line}). 
\end{enumerate}
Then there exists an isometric $C^{k+1}$-diffeomorphism $\Phi:(M' \times \R, h+dt^2) \to (M,g)$, where $M'$ is a smooth manifold with a Riemannian metric $h \in C^k$ which satisfies $\Ric_h \geq 0$.
\end{Theorem}

The second rigidity result we prove concerns characterizations of flatness of low regularity metrics. This result is independent of the signature, so we formulate it for general semi-Riemannian manifolds:

\begin{Theorem}[Flatness criterion for low regularity semi-Riemannian metrics]
    Let $M$ be a smooth manifold and $g$ a semi-Riemannian metric tensor on $M$ of constant signature $(l_1,l_2)$ and of regularity $C^k$, $k \geq 1$. Then the following are equivalent:
    \begin{enumerate}
        \item $\Riem_g = 0$ in $C^{k-2}(\cT^1_3 M)$ (where $C^{k-2} = \mathcal{D}^{'(1)}$ if $k=1$).
        \item $(M,g)$ is $C^k$-frame flat, i.e.\ in a neighborhood of every point there exists a parallel orthonormal frame of regularity $C^k$.
        \item $(M,g)$ is $C^{k+1}$-coordinate flat, i.e.\ each point in $M$ is contained in a neighborhood $U$ for which there exists an isometric $C^{k+1}$-diffeomorphism $\phi:U \to \phi(U) \subseteq \R^{l_1,l_2}$ onto an open subset of $\R^{l_1,l_2}$, the semi-Euclidean space of signature $(l_1,l_2)$.
    \end{enumerate}
\end{Theorem}

The article is structured as follows: In Subsection \ref{Subsection: NotationConventions}, we collect the notations and conventions used throughout. In Section \ref{Section: Preliminary} we review material concerning distributions on manifolds (Subsection \ref{Subsection: distributions}), distributional curvature and regularization (Subsection \ref{Subsection: C1distribcurvat}) and local Sobolev spaces (Subsection \ref{Subsection: localSobolevspaces}). In Section \ref{Section: Rigidity}, we establish the aforementioned rigidity results for low regularity metrics: The Cheeger-Gromoll splitting theorem in Subsection \ref{Section: C1splitting} and the flatness criterion in Subsection \ref{Subsection: Flatness}. As an auxiliary result for the proof of the Cheeger-Gromoll theorem, we give a low regularity version of the Bochner-Weitzenböck formula in Subsection \ref{Subsection: BochnerWeitzenbock}. Section \ref{Section: Syntheticdistributional} is concerned with a discussion of the relationship between synthetic and distributional curvature bounds in low regularity, complementing the results established in \cite{mondino2024equivalence}. In Subsection \ref{Subsection: CompatibilitybetweennotionsofSobolev} we discuss the various notions of Sobolev spaces available in the setting of low regularity metrics (see also the recent preprint \cite{ambrosio2024metric} for more on this topic) and in Subsection \ref{Subsection: Distributional -> synthetic} we give a simpler proof of the fact that distributional Ricci curvature bounds imply the $\mathsf{RCD}$-condition (established in \cite{mondino2024equivalence}) for locally Lipschitz metric tensors using approximation-based techniques. In Subsection \ref{Subsection: Consequencesofequivalence} we collect some consequences for the geometry of low regularity Riemannian metrics that are immediate due to the equivalence between distributional Ricci bounds and the $\mathsf{RCD}$-condition. Finally, in Section \ref{Section: Conclusionandoutlook}, we give a summary of our work and an outlook on possible further related lines of research.

\subsection{Notation and conventions}\label{Subsection: NotationConventions}

Given a function (or distributional) space $\mathcal{F}$ and a bundle $E$ over $M$, we will denote by
\begin{equation*}
    \mathcal{F} (E)
\end{equation*}
the space of sections $s:M\to E$ of regularity $F$. For example
\begin{align*}
    C^\infty(T^*M)&:=\Omega^1(M),\\
    C^1_c(TM)&:=\Gamma_c^{(1)}(M),\\
    \D'(\cT^r_s M)&:= \{ \text{Distributional }(r,s)\text{-tensors}\},\\
    W^{1,2}_\loc(TM)&:=\{W^{1,2}_\loc \text{-vector fields}\},\\
    &\text{etc...}
\end{align*}
Following \cite{petersen2006riemannian},
for the differential operators on a (semi-) Riemannian manifold $(M,g)$ we are going to use below we shall adhere to the following conventions. Let $V,W,X$, $Y_1,...,Y_s$ be vector fields, $\omega$ be a $1$-form, $\eta$ a $2$-form, $A$ a tensor field then
\begin{align*}
    \Div X &= \frac{1}{\sqrt{\abs{g}}}\partial_\alpha({\sqrt{\abs{g}}} X^\alpha ),\\
    \Delta f &= \Div(\nabla f),\\
     \nabla^2_{V,W}A (Y_1,...,Y_s) &= (\nabla\nabla A)(V,W,Y_1,...,Y_s) \\
    &=\nabla_V\nabla_W A(Y_1,...,Y_s) - \nabla_{\nabla_V W}A(Y_1,...,Y_s), \\
    \Hess f &= \nabla(\de f),\\
    \de \omega(V,W) &= [\nabla\omega(V,W)- \nabla \omega(W,V)].
\end{align*}
The codifferential (or adjoint differential) $\delta$ is defined via adjointness: Given a sufficiently regular $k$-form $\eta$, $\delta\eta$ is the $(k-1)$-form such that
\begin{align*}
    \int g( \eta,\de \omega) \,\de\Vol = \int g (\delta\eta,\omega) \,\de\Vol
\end{align*}
for all smooth compactly supported $(k-1)$-forms $\omega$. For convenience we will only write in explicit form the operator for $k=1,2$ and with respect to some orthonormal frame (ONF) $\{e_i\}_i$, which consists of vector fields that are at least locally Lipschitz under our assumptions:
\begin{gather*}
    \delta \omega = -\sum_i \nabla_{e_i}\omega(e_i),\\
    \delta\eta(V) = -\sum_i \nabla_{e_i}\eta(e_i,V). 
\end{gather*}
The above definitions can be generalized to distributional forms assuming the right hand sides are well-defined. 

Similarly, we define by adjointness the adjoint connection $\nabla^*$ of a sufficiently regular  $(r,s)$-tensor field $A$ as the $(r,s-1)$-tensor field $\nabla^*A$ such that
\begin{equation*} 
    \int g( A,\nabla B) \,\de\Vol = \int g (\nabla^* A, B) \,\de\Vol
\end{equation*}
for all smooth compactly supported $(r,s-1)$-tensor fields $B$. Given an ONF $\{e_i\}_i$ this reduces to
\begin{equation*}
    \nabla^*A (Y_1,...,Y_s)= -\sum_i \nabla _{e_i}A(e_i,Y_1,...,Y_s)
\end{equation*}

In order to minimize  ambiguities, we will denote by $g(A,B)$ the function (or distribution)
given by the scalar product between two $(r,s)$-tensor fields $A$ and $B$. The angled brackets notation will be reserved for the evaluation of distributions: for $T\in \D'(M)$ and $\mu \in C^\infty_c(\Vol\, M) $, we denote $\langle T, \mu\rangle := T(\mu)$.

\section{Preliminary material}\label{Section: Preliminary}
In this first preparatory section, we collect background material that will be useful for later parts of the paper. We recall the notion of distributional curvature tensors for low regularity metrics as well as a generic regularization procedure via manifold convolution which will be useful to link conditions on distributional or non-smooth objects to approximate conditions on smooth approximating ones.

\subsection{Distributions and distributional connections}
\label{Subsection: distributions}

Let us give a very brief recap of the definitions of distributions and distributional connections on smooth manifolds. We refer to \cite{LeFlochMardare} and \cite{gkos2001geometricgeneralized} for a more detailed discussion.

Let $M$ be a manifold of dimension $n$ with a smooth atlas $\{(U_\alpha,\psi_\alpha)\}_\alpha$ endowed with a continuous semi-Riemannian metric $g$. Distributions on $M$ are defined as elements of the topological dual of the space of compactly supported smooth sections of the volume bundle (\emph{volume densities)},
\begin{align*}
    \D'(M):= \left[ C_c^\infty(\Vol\, M)\right]',
\end{align*}
where $\Vol \,M$ is the vector bundle of rank $1$ defined by the transition functions
\begin{align*}
    \Psi_{\alpha\beta} (x) = \abs{\det \de(\psi_\beta\circ \psi_\alpha^{-1}) (\psi_\alpha(x))}.
\end{align*}
On a chart we can always represent a smooth volume density as
\begin{align*}
    u \abs{\de x^1 \wedge...\wedge \de x^n} \quad\text{with } u\in C^\infty(U_\alpha).
\end{align*}
If $M$ is oriented, $\Vol\, M$ coincides with $\Lambda^n T^*M$. In general, continuous sections of $\Vol\, M$ are Radon measures and can thus integrate functions on $M$. Any $f\in L^1_{\mathrm{loc}}(M)$ then can naturally be viewed as an element of $\D'(M)$ via integration:
\[
C^\infty_c(\Vol\, M)\ni \mu \longmapsto {\langle f,\mu\rangle }:= \int_M f \,\de\mu.
\]
The presence of a metric $g$ gives us also a way to view any function $f\in C^\infty_c(M)$ (or more  generally $f$ measurable) as a volume density section $f\Vol$, where
\begin{align*}
    \Vol = \sqrt{\abs{\det g}}\abs{\de x^1\wedge...\wedge \de x^n}\in C^0 (\Vol\, M)
\end{align*}
is the canonical volume density induced by $g$. It is important to note, however, that $f\Vol$ may not be a smooth section anymore, thus the evaluation of a distribution on $f\Vol$ would have to be justified in each case. 

We can more generally define distributional $(r,s)$-tensors:
\begin{equation*}
    \D'(\cT^r_s M) := 
    \left[C_c^\infty(\cT^s_r M \otimes \Vol\,M)\right]' \cong \D'(M)\otimes_{C^\infty(M)} C^\infty(\cT^r_s M).
\end{equation*}
The brackets notation $\langle \cdot , \cdot \rangle$ will be also used for distributional tensors: for $T\in \D'(T^r_s M), \omega\in C^\infty_c(T^s_r M \otimes \Vol\, M)$ we write
\begin{equation*}
    \langle T, \omega\rangle := T(\omega).
\end{equation*}

\begin{Definition}
    We call a map $\nabla:C^\infty(TM)\times C^\infty(TM) \to \D'(TM)$, 
    \emph{a distributional connection} if it satisfies
     for all $X,X',Y,Y'\in C^\infty(TM)$ and $f\in C^\infty(M)$
    \begin{enumerate}
        \item $\nabla_{fX+X'}Y = f\nabla_{X}Y+\nabla_{X'}Y$,
        \item $\nabla_X(Y+Y')= \nabla_X Y + \nabla_X Y'$,
        \item $\nabla_X (fY) = X(f) Y + f \nabla_X Y$.
    \end{enumerate}
    Given a functional (or distributional) space $\mathcal{F}$, we say a distributional connection $\nabla$ is an \emph{$\mathcal{F}$-connection} if the image of $\nabla$ is contained in $\mathcal{F}(TM)$.
\end{Definition}
Given a $C^{k,\alpha}_\loc$ (resp. $C^0\cap W^{1,p}_\loc$ with $p\geq 2$) semi-Riemannian metric $g$, the Koszul formula
\begin{align*}
    2g(\nabla_X Y,Z) = &X(g(Y,Z)) + Y(g(X,Z)) - Z(g(X,Y))\\
    &-g(X,[Y,Z]) + g(Y,[Z,X]) + g(Z,[X,Y])
\end{align*}
allows us to identify the unique Levi-Civita $C^{k-1,\alpha}_\loc$-connection (resp. $L^p_\loc$-connection) associated to $g$. The local expression of the covariant derivative
\begin{gather*}
    \nabla_X Y = (X^\beta \partial_\beta Y^\alpha + \Gamma^\alpha_{\beta\gamma}X^\beta Y^\gamma)\partial_\alpha,\\
    \text{with }\Gamma^\alpha_{\beta\gamma}=\frac{1}{2}g^{\alpha\lambda}(\partial_\gamma g_{\lambda \beta }+ \partial_\beta g_{\gamma \lambda }- \partial_\lambda g_{\beta \gamma}),
\end{gather*}
highlights the possibility to uniquely extend the Levi-Civita connection associated to $g$ to lower regularities of the vector fields $X,Y$. Indeed, given three functional (or distributional) spaces $\mathcal{F},\mathcal{G},\mathcal{H}\subseteq \D'(M)$ such that locally $X^\beta \partial_\beta Y^\alpha, \Gamma^\alpha_{\beta\gamma}X^\beta Y^\gamma \in \mathcal{H}$ for all $X\in \mathcal{F}(TM),Y\in \mathcal{G}(TM)$, we can uniquely extend the distributional Levi-Civita connection to a map
\begin{align*}
    \nabla: \mathcal{F}(TM)\times \mathcal{G}(TM) \longrightarrow \mathcal{H}(TM).
\end{align*}

Suppose that, given a metric $g$, the associated Levi-Civita connection
\begin{align*}
    \nabla: C^\infty(TM)\times C^\infty(TM) \longrightarrow \mathcal{F}(TM)
\end{align*}
can be extended to (for suitable $\mathcal{F}$)
\begin{align*}
    \nabla: C^\infty(TM)\times \mathcal{F}(TM) \longrightarrow \D'(TM).
\end{align*}
Then we can define the Riemann and Ricci curvature tensors as distributional tensors
\begin{align*}
    \Riem \in \D'(\cT^1_3 M),\qquad \Ric\in \D'(\cT^0_2 M)
\end{align*}
as follows: Given $X,Y,Z\in C^\infty(TM)$, a local frame $F_i$ in $C^\infty(TM)$ and its dual frame $F^i\in C^\infty(T^*M)$, we define
\begin{align*}
    \Riem(X,Y,Z)&:= \nabla_X \nabla_Y Z - \nabla_Y \nabla_X Z - \nabla_{[X,Y]} Z,\\
    \Ric(X,Z)&:= [\Riem (X,F_i) Z](F^i).
\end{align*}
For example, if $g\in C^0(T^0_2 M) \cap W^{1,2}_\loc(T^0_2 M)$, then $\mathcal{F}= L^2_\loc$.

\subsection{Distributional curvature and regularization}\label{Subsection: C1distribcurvat}

Let us recall the general regularization scheme for non-regular tensors (in particular, metrics) based on local convolution that we shall employ throughout this article. A reference for this material that is suitable for our purposes is \cite[Sec.\ 2]{KOSS22C1HawkingPenrose}.

\begin{Definition}
    Fix a non-negative convolution kernel $\{\rho_\eps\}_{\eps>0}$ on $\R^n$ and a countable atlas $\{(U_\alpha, \psi_\alpha) \}_{\alpha\in \N}$ with relatively compact $U_\alpha$, a subordinate smooth partition of unity $\{\xi_\alpha\}_\alpha$ and functions $\chi_\alpha \in C^\infty_c(M)$ with $0\leq \chi_\alpha\leq 1$ and $\chi_\alpha \equiv 1$ on a neighborhood of $\supp \xi_\alpha $. Then for any $T\in \D'(\cT^r_s M)$ and $\eps>0$ we define the smooth $(r,s)$-tensor
\begin{equation*}
    T \star_M \rho_\eps := \sum_{\alpha\in\N}\chi_\alpha (\psi_\alpha)^*( ((\psi_\alpha)_*(\xi_\alpha T) )* \rho_\eps),
\end{equation*}
where the convolution $((\psi_\alpha)_*(\xi_\alpha T) )* \rho_\eps$ is to be understood in the component-wise sense, recalling that each component is in $\D'(\psi_\alpha (U_\alpha))$.
\end{Definition}

\begin{Proposition}\label{prop:standard_properties_of_starM}
    The following statements hold:
    \begin{enumerate}
        \item If $T \in \D'(\cT^0_2 M)$ is symmetric and satisfies $T\geq 0$, then $T \star_M \rho_\eps\ge 0$ for all $\eps>0$.
        \item If $T\in \D'(\cT^r_s M)$, then 
        \begin{equation*}
            \langle T \star_M \rho_\eps , \mu \rangle \to \langle T,\mu \rangle \qquad \forall \mu\in C^\infty_c(T^s_r M\otimes \Vol\, M).
        \end{equation*}
        \item If $T\in C^k$, for some $k\in \N$, then $T \star_M \rho_\eps \to T$ in $C^k_\loc$.
    \end{enumerate}
\end{Proposition}

\begin{proof} All of these claims follow from the definition of $\star_M$ and standard
properties of convolution on $\R^n$.    
\end{proof}

\begin{Remark}
\label{Remark: Approximatingmetrics}
    Given a continuous semi-Riemannian metric $g$ on $M$, by Proposition \ref{prop:standard_properties_of_starM} (iii), $g\star_M \rho_\eps \to g$ locally uniformly for $\eps\to 0$. Thus for any $K\Subset M$ there is an $\eps_K>0$ such that for all $\eps\in (0,\eps_K]$
we have:
\begin{align}
g\star_M\rho_\eps|_K \text{ is a semi-Riemannian metric of the same signature as $g$}. \label{geps1} 
\end{align}
Since \eqref{geps1} is stable with respect to decreasing $K$ and $\eps$, we may apply \cite[Lem.\ 4.3]{HKS_wave} to obtain a smooth
map $(\eps,x) \mapsto g_\eps(x)$ that possesses this property globally on $(0,1] \times M$ (i.e., each $g_\eps$ is a semi-Riemannian metric on $M$ of the same signature as $g$) and such that, in addition, for every $K\comp M$ there exists some $\eps_K>0$ with $g_\eps(x) = g\star_M\rho_\eps(x)$ for all $(\eps,x)\in (0,\eps_K]\times K$. 

In the special case (most relevant for our purposes here) of $g$ being Riemannian, \eqref{geps1} is automatically satisfied globally (since local convolution with $\rho_\eps \ge 0$ preserves positive definiteness). In this signature, we additionally have that for each $K\Subset M$ there is an $\eps_K>0$ such that for all $\eps\in (0,\eps_K]$
\begin{equation}\label{geps2} 
\frac{1}{2}\abs{v}_g \le \abs{v}_{g\star_M\rho_\eps} \le 2\abs{v}_g \quad \forall v\in TM|_K,
\end{equation}
and again applying \cite[Lem.\ 4.3]{HKS_wave} we construct Riemannian metrics $g_\eps$ satisfying 
\begin{equation}\label{geps3} 
\frac{1}{2}\abs{v}_g \le \abs{v}_{g_\eps} \le 2\abs{v}_g \quad \forall v\in TM,
\end{equation}
Then in particular, for any  $C^1$-curve $\gamma$ connecting two points in $M$, for its lengths with respect to $g$ resp.\ $g_\eps$ we have $\frac{1}{2}L_g(\gamma) \le L_{g_\eps}(\gamma) \le 2L_g(\gamma)$, from which it follows that for the corresponding Riemannian distances
we have:
\begin{equation}\label{eq:distances_g_geps}
    \frac{1}{2} d_g(x,y) \le d_{g_\eps}(x,y) \le 2d_g(x,y) \qquad \forall x, y \in M.
\end{equation}
In particular, for the corresponding metric balls this gives:
\begin{equation}\label{eq:balls_g_geps}
B^{d_g}_{r/2}(x) \subseteq B^{d_{g_\eps}}_r(x) \subseteq B^{d_g}_{2r}(x)
\qquad (x\in M, r>0).
\end{equation}
\end{Remark}

Let us fix a sequence $\eps_i \searrow 0$. We will denote $g_i = g_{\eps_i}$, $\rho_i:=\rho_{1/i}$ and all the related objects will have the subscript index $i$, e.g.\ we shall write $d_i$ for the Riemannian distances with respect to $g_i$, $\Vol_i$ for the $g_i$-volume measures,  etc. Quantities without this subscript will always refer to the metric $g$ itself, e.g., $\Ric$ is the Ricci tensor of $g$. By the above construction, for compact subset $K\Subset M$, $g_i|_K = (g\star_M \rho_i)|_K$ for all large $i$.

The following result collects some known convergence properties of $\star_M$-regularizations of metrics that we shall make use of below.
\begin{Proposition}\label{prop:LpRicConvergence}
    Let $(M,g)$ be a
    $C^\infty$-manifold with semi-Riemannian metric 
    $g\in C^{0,1}_\loc(\cT^0_2 M)$. 
    Then
    \begin{itemize}
    \item[(i)] ${\Ric_i - \Ric \star_M \rho_{i}}\to 0 
    \quad \text{in }L^p_\loc(\cT^0_2 M) \text{ for all }p\in [1,\infty).$ 
    If $g\in C^1(\cT^0_2 M)$, then the convergence is even locally uniform.
    \item[(ii)] For $X, Y \in C^{\infty}(TM)$, 
    $\Ric_i(X,Y) - \Ric(X,Y)\star_M \rho_{i} \to 0 \quad \text{in } L^p_{\mathrm{loc}}(M)$.
    \end{itemize}
    \begin{proof} All these claims, except for the stronger convergence result in (i), are shown in \cite{CGHKS24}. For the latter we refer to \cite[Lem.\ 4.5]{graf2020singularity}. 
    \end{proof}   
\end{Proposition}

\subsection{Local Sobolev spaces of integer order}
\label{Subsection: localSobolevspaces}
Let us recall basic properties of local Sobolev spaces on domains of $\R^n$. We refer to \cite{antbur2005sobloc} for more details.

Given an open domain $\Omega\subseteq \R^n$, in the space of distributions $\D '(\Omega)$ we can identify, for all $m\in \N\setminus \{0\}$ and $p\in(1,\infty)$, the subspaces
\begin{equation*}
	W^{-m,p}(\Omega):= (W_0^{m,\frac{p}{p-1}} (\Omega))' \cong \left\{\sum_{\abs{\alpha}\leq m} D^\alpha f_\alpha \in \D'(\Omega) : f_\alpha \in L^p(\Omega) \right\}\subseteq \D '(\Omega).
\end{equation*}
An important property for our investigations in this paper is the possibility to multiply distributions in $W^{-m,p}(\R^n)$ with functions in $W^{h,q}(\R^n)$, for appropriate $h\in \N$ and $q\in (1,\infty)$. In particular we will use the following fact.
\begin{Remark}\label{rk:W-1pMultipl}
    For distributions in $W^{-1,1+\frac{1}{n}}(\R^n)$, the multiplication with smooth functions
    \begin{align*}
        C^\infty(\R^n) \times W^{-1,1+\frac{1}{n}}(\R^n) \longrightarrow \D'(\R^n)
    \end{align*}
    can be extended to a continuous bilinear map
    \begin{align*}
        W^{1,n+1}(\R^n) \times W^{-1,1+\frac{1}{n}}(\R^n) \longrightarrow W^{-1,1+\frac{1}{n}}(\R^n).
    \end{align*}
    A proof can be found in \cite[Thm.\ A.1]{behholst2021sobolevmultiplication}. 

\end{Remark}
We also recall the following definition, for $k\in\Z$ and $p\in (1,\infty)$,
\begin{equation*}
    W^{k,p}_\loc(\Omega):=\{f\in \D'(\Omega): \forall \varphi\in C^\infty_c(\Omega),\, \varphi f \in W^{k,p}(\Omega) \}.
\end{equation*}
In order to extend these definitions to Riemannian manifolds, we will simply ask for a characterization through charts, given a smooth atlas $\{(U_\alpha,\psi_\alpha) \}_{\alpha\in \N} $.
\begin{equation*}
    W^{k,p}_\loc(M):=\{f\in \D'(M): \forall \alpha \in \N,\, (\psi_\alpha)_* f \in W^{k,p}_\loc(\psi_\alpha(U_\alpha))\},
\end{equation*}
where the \textit{pushforward} $(\psi_\alpha)_* f\in \D'(\psi_\alpha(U_\alpha))$ is defined, given $\psi_\alpha=(x_1,...,x_n)$, as
\begin{align*}
    (\psi_\alpha)_* f:\quad \vphi \longmapsto f\left(\vphi\circ \psi_\alpha \abs{\de x^1\wedge...\wedge\de x^n}\right).
\end{align*}
We also note the following equivalent characterization.
\begin{gather*}
    W^{k,p}_\loc(M):=\{f\in \D'(M):\forall \alpha \in \N,\, \forall \chi\in C^\infty_c(U_\alpha),\, (\psi_\alpha)_*(\chi f) \in W^{k,p} (\R^n) \}.
\end{gather*}

We can further extend these spaces of scalar distributions to (sections of) tensor bundles via tensorization.
\begin{gather*}
    W^{k,p}_\loc(\cT^r_s M):=W^{k,p}_\loc(M)\otimes_{C^\infty (M)} C^\infty (\cT^r_s M).
\end{gather*}
In simpler terms, $T\in W^{k,p}_\loc(\cT^r_s M)$ whenever it is an $(r,s)$-tensor with coefficients in $W^{k,p}_\loc$.\\
From now on we will use the following notation:
\begin{align*}
	W^{m,q^-}(X):= \bigcap_{1\leq p <q} W^{m,p}(X)
\end{align*}
where $q\in[1,\infty]$, $m\in \N$ and $X$ could either be a Euclidean space, an open domain, a Riemannian manifold or a tensor bundle. Analogous definitions hold for $L^{q^-}(X)$ and $W^{m,q^-}_\loc(X)$. Observe that $C^{0,1}_{\loc}(\cT^r_s M)) = W^{1,\infty}_{\loc}(\cT^r_s M)) \subseteq W^{1,\infty^-}_{\loc}(\cT^r_s M))$ by Rademacher's theorem.

\section{Rigidity for non-smooth (semi-)Riemannian manifolds}
\label{Section: Rigidity}
In this main section of the article, we first establish a Bochner-Weizenböck identity where we allow both the metric and the vector fields to be non-smooth, cf.\ Theorem \ref{thm:BochWeitz}. This is then used to prove a version of the Cheeger-Gromoll splitting theorem for Riemannian metrics of $C^1$-regularity (see Theorem \ref{thm:SplittingTheorem}), where we are able to establish $C^2$-regularity of the splitting isometry. We clarify the relationship between the metric regularity and the regularity of the isometry in Theorem \ref{thm: higherregularitysplitting}. In the last subsection, we prove the equivalence of the usual notions of flatness (vanishing of the Riemann tensor, existence of local parallel orthonormal frames, and local isometry to flat space) for metrics of regularity $C^1$. Here, no properties of the Riemannian signature are used, we work in the general semi-Riemannian setting. 

\subsection{The Bochner-Weitzenb\"ock identity for $C^{0,1}_\loc$-Riemannian metrics}
\label{Subsection: BochnerWeitzenbock}
In this subsection, we present a low regularity version of the generalized Bochner-Weitzenböck formula, which extends \cite[Prop.\ 3.9]{mondino2024equivalence} (in the case of Lipschitz metrics) to non-smooth vector fields. We need a preparatory Lemma on covariant differentiation of non-regular vector fields.
\begin{Lemma}\label{lem:W-1pCovar}
	Let $(M,g)$ be a Riemannian manifold with $g \in C^{0,1}_{\loc}(\cT^0_2 M )$. For $V\in L^{\infty^-}_\loc(TM)$ and $W\in W^{1,\infty^-}_\loc(TM)$, it holds that
	\begin{equation*}
		\nabla_V W \in L^{\infty^-}_\loc(TM),\qquad  \nabla_W V \in W^{-1,1+ \frac{1}{n}}_\loc(TM).
	\end{equation*}
 Moreover, for $X\in W^{-1,1+ \frac{1}{n}}_\loc(TM)$ and $Y\in W^{1,\infty^-}_\loc(TM)$, their scalar product is a well defined as a distribution
 \begin{equation*}
     g( X, Y ) \in W^{-1,1+ \frac{1}{n}}_\loc(M).
 \end{equation*}
\end{Lemma}
 \begin{proof} Since these claims are local, we may without loss of generality suppose that 
 all vector fields are compactly supported and $M=\R^n$.
	Recall that in smooth local coordinates
	\begin{align*}
		(\nabla_V W )^\beta= V^\alpha(\partial_\alpha W^\beta + \Gamma_{\alpha\gamma}^\beta W^\gamma).
	\end{align*}
	Observing that $\Gamma_{\alpha\gamma}^\beta \in L^\infty_{\mathrm{loc}}(\R^n)$ and that $L^{\infty^-}(\R^n)$ is closed with respect to the pointwise product, we conclude that $(\nabla_V W)^\beta \in L^{\infty^-}(\R^n)$.\\
	Similarly, applying Remark \ref{rk:W-1pMultipl} to	\begin{align*}
		(\nabla_W V )^\beta= W^\alpha(\partial_\alpha V^\beta + \Gamma_{\alpha\gamma}^\beta V^\gamma).
	\end{align*}
we see that $W^\alpha\partial_\alpha V^\beta \in W^{-1,1+\frac{1}{n}}(\R^n)$, while $W^\alpha\Gamma_{\alpha\gamma}^\beta V^\gamma \in L^{\infty^-}(\R^n)$. We conclude by recalling that
\begin{equation*}
    L^{\infty^-}(\R^n)\subseteq L^{1+\frac{1}{n}}(\R^n) \cong (L^{n+1}(\R^n))' \subseteq (W^{1, n+1}(\R^n))' = W^{-1,1+\frac{1}{n} }(\R^n).
\end{equation*}
For the second statement observe that $g( X,Y ) = X^\alpha Y^\beta g_{\alpha\beta}$. Here $Y^\beta\in W^{1,\infty^-}(\R^n)$ and $g_{\alpha\beta}\in W^{1,\infty}_{\mathrm{loc}}(\R^n)$, thus $Y^\beta g_{\alpha\beta}\in W^{1,\infty^-}(\R^n)$. Thus another appeal to Remark \ref{rk:W-1pMultipl} concludes the proof.
\end{proof}
Let us denote here the Hodge Laplacian on 1-forms by $\Delta_H = \delta \de + \de\delta$. We will also extend the operator $\Delta_H$ to vector fields in the following way:
\begin{equation*}
	\Delta_H X = (\Delta_H X^\flat)^\sharp,
\end{equation*}
where
\begin{equation*}
    \flat: TM \longrightarrow T^*M,\qquad \sharp: T^*M \longrightarrow TM
\end{equation*}
are the Riesz musical isomorphisms, which are local Lipeomorphisms if the metric is locally Lipschitz. For the Laplace-Beltrami operator on scalar functions we will follow the convention $\Delta f = -\delta df$.

\begin{Theorem}[Generalized Bochner-Weitzenböck identity]\label{thm:BochWeitz}
    Let $(M,g)$ be a Riemannian manifold with metric $g \in C^{0,1}_{\loc}(\cT^0_2 M)$. Then for all vector fields $X\in W^{1,\infty^-}_\loc(TM)$ the following identity holds in $W^{-1,1+\frac{1}{n}}_{\loc}(M) \subseteq \mathcal{D}'(M)$: 
	\begin{equation}\label{eq:BochWeitz}
    	\Ric(X,X)= \Delta \frac{\abs{X}^2}{2} - \abs{\nabla X}^2 +
            g( \Delta_H X,X).
	\end{equation}
\end{Theorem}
\begin{proof}
    Let us argue locally by fixing an ONF $\{e_i\}_i$, which consists of $C^{0,1}_\loc$-vector fields. Lemma \ref{lem:W-1pCovar} ensures that all objects in this proof are well-defined.\\
    We first compute $\de\delta X^\flat$ explicitly, evaluated on an arbitrary $W^{1,\infty^-}_\loc$-vector field $Y$. Observe that since $g\in W^{1,\infty}_\loc(\cT^0_2 M)$ and $X\in W^{1,\infty^-}_\loc(TM)$, then $X^\flat \in W^{1,\infty^-}_\loc (T^*M)$. 
    \begin{align*}
        \de\delta X^\flat (Y)&= \nabla_Y\Big(-\sum_i \nabla_{e_i} X^\flat (e_i)\Big) = -\sum_i \big(\nabla_Y (\nabla_{e_i} X^\flat)(e_i) +\nabla_{e_i} X^\flat (\nabla_Y e_i) \big)\\
        &=-\sum_i \big(\nabla^2_{Y,e_i}X^\flat(e_i) + \nabla_{e_i} X^\flat (\nabla_Y e_i) +\nabla_{\nabla_Y e_i} X^\flat (e_i)\big).
    \end{align*}
    Now we claim that 
    \begin{align*}
        \sum_i \big(\nabla_{e_i}  X^\flat (\nabla_Y e_i) +\nabla_{\nabla_Y e_i} X^\flat (e_i)\big) =0 .
    \end{align*}
    Indeed up to $L^1_\loc$-linearity observe that for $\nabla X^\flat = \alpha \otimes \beta$
    \begin{align*}
        &\sum_i \alpha\otimes\beta (e_i, \nabla_Y e_i) + \alpha\otimes\beta (\nabla_Y e_i , e_i) = \sum_i \alpha\otimes\beta ( \nabla_Y (e_i\otimes e_i))\\
        &= \nabla_Y\left[\sum_i \alpha\otimes\beta (e_i, e_i) \right] -\sum_i(\nabla_Y \alpha)\otimes \beta(e_i,e_i) -\sum_i \alpha\otimes(\nabla_Y \beta) (e_i,e_i)\\
        &=\nabla_Y (\alpha (\beta^\sharp)) - (\nabla_Y \alpha)(\beta^\sharp) - \alpha(\nabla_Y \beta^\sharp)=0.
    \end{align*}
    Let us denote by $S^T$ the transposed tensor of the $(0,2)$-tensor $S$, i.e. $S^T(V,W):= S(W,V)$. We can then compute
    \begin{align*}
        \delta\de X^\flat (Y) &= \delta\left[(V,W)\mapsto \nabla X^\flat (V,W) - \nabla X^{\flat}(W,V) \right] (Y)\\
        &=-\sum_i \nabla_{e_i} ( \nabla X^\flat - \nabla X^{\flat T})(e_i, Y) \\
        &=-\sum_i \nabla_{e_i} (\nabla X^\flat)(e_i, Y) + \sum_i \nabla_{e_i} ( \nabla X^{\flat})(Y,e_i) \\
        &= -\sum \nabla^2_{e_i,e_i} X^\flat (Y) + \sum_i \nabla^2_{e_i,Y} X^\flat (e_i).
    \end{align*}
    Adding together the two previous results, with $Y=X$, we get
    \begin{align*}
        \Delta_H X^\flat (X) = - \sum_i \nabla^2_{e_i,e_i}X^\flat (X) + \sum_i (\nabla^2_{e_i,X} X^\flat (e_i) - \nabla^2_{X,e_i} X^\flat (e_i)),
    \end{align*}
    which can be rewritten as
    \begin{align*}
        g( \Delta_H X, X) = -g( \sum_i \nabla^2_{e_i,e_i}X ,X) + \Ric(X,X).
    \end{align*}
    Indeed it holds that $\nabla V^\flat = (\nabla V)^\flat$ for $V\in L^{\infty^-}_\loc(TM)$, which is to be understood as a distributional identity. This can be verified in coordinates:
\begin{align*}
    &[\nabla V^\flat - (\nabla V)^\flat]_{\alpha\beta}\\
    &=\partial_\alpha (g_{\beta i} V^i) - \Gamma^j_{\alpha\beta}g_{ji}V^i - g_{\beta i}\partial_\alpha V^i -\Gamma_{\alpha i}^j g_{\beta j}V^i\\
    &= \frac{V^i}{2}\left( 2\partial_\alpha g_{\beta i} -\partial_\alpha g_{\beta i} -\partial_\beta g_{\alpha i} +\partial_i g_{\alpha\beta } -\partial_\alpha g_{\beta i} -\partial_i g_{\alpha \beta} +\partial_\beta g_{\alpha i}\right)\\
    &=0.
\end{align*}
Apart from $g_{\beta i} \partial_\alpha V^i$, which is admissible by Remark \ref{rk:W-1pMultipl}, all the above are pointwise multiplications of functions.\\
    Now observe that
    \begin{align*}
        g( \sum_i \nabla^2_{e_i,e_i}X ,X) 
        &= \sum_i (g(  \nabla_{e_i}\nabla_{e_i}X,X) -g(\nabla_{\nabla_{e_i}e_i}X,X))\\
        &=\sum_i (\nabla_{e_i}g(\nabla_{e_i}X,X) - g( \nabla_{e_i}X, \nabla_{e_i}X) -g(\nabla_{\nabla_{e_i}e_i}X,X))\\
        &=\sum_i \nabla_{e_i}\left[V\mapsto g(\nabla_V X,X )\right] (e_i) - \abs{\nabla X}^2\\
        &=\sum_i \nabla_{e_i}(\de \frac{\abs{X}^2}{2}) (e_i) - \abs{\nabla X}^2\\
        &= \Delta \frac{\abs{X}^2}{2} - \abs{\nabla X}^2.
    \end{align*}
\end{proof}

\begin{Remark}
    Observe that by polarization of (\ref{eq:BochWeitz}) we obtain a complete characterization of $\Ric(X,Y)$ for $X,Y \in W^{1,\infty^-}_\loc(TM)$:
    \begin{align*}
        \Ric(X,Y) &= \frac{1}{4}\left(\Ric(X+Y,X+Y)- \Ric(X-Y,X-Y)\right)\\
        &= \Delta \frac{g( X,Y)}{2} - g( \nabla X ,\nabla Y) +\frac{1}{2}\left(g( \Delta_H X,Y) + g( \Delta_H Y,X)\right).
    \end{align*}
    Additionally, in the proof of Theorem \ref{thm:BochWeitz} we obtained the low-regularity generalization of another well-known formula:
    \begin{equation*}
        \Ric(X,\cdot) = \Delta_H X^\flat - \nabla^*\nabla X^\flat \qquad \text{for all }X\in W^{1,\infty^-}_\loc(TM)
    \end{equation*}
    where $\nabla^*\nabla X^\flat = -\sum_i \nabla^2_{e_i,e_i} X^\flat$ is the so-called \emph{Bochner} or \emph{connection Laplacian}.
\end{Remark}

\begin{Remark}
    Though not necessary for the application of the result in the later parts of the paper, Theorem \ref{thm:BochWeitz} can be expected to generalize to the case $g\in C^0(\cT^0_2 M)\cap W^{1,2}_\loc(\cT^0_2 M)$, $X\in C^0(T M)\cap W^{1,2}_\loc(T M)$, as an identity in $\D'(M)$. Once one proves all objects in the proof are well-defined, the proof can be followed as is. For instance, one could use the multiplication $W^{1,2}(\R^n)\times W^{-1,2}(\R^n) \to W^{-n,2}(\R^n)$, defined in \cite[Thm. 8.1]{behholst2021sobolevmultiplication} to justify the extention of the Levi-Civita connection associated to $g$ to
    \begin{align*}
        \nabla: (C^0(TM) \cap W^{1,2}_\loc(TM))\times L^2_\loc(TM)\longrightarrow \D'(TM),
    \end{align*}
    and similarly
    \begin{align*}
        g(\nabla_\cdot \cdot , \cdot): (C^0(TM)\cap W^{1,2}_\loc(TM))\times L^2_\loc(TM) \times (C^0(TM)\cap W^{1,2}_\loc(TM))\longrightarrow \D'(M).
    \end{align*}
\end{Remark}

\subsection{The Cheeger-Gromoll splitting theorem for $C^1$-Riemannian metrics}\label{Section: C1splitting}

Let $(M,g)$ be a complete\footnote{Whenever we say that a Riemannian manifold $(M,g)$, with $g$ of lower regularity, is \emph{complete}, we mean that the induced metric space $(M,d_g)$ is complete. The connection to geodesic completeness is established in Theorem \ref{thm:C1hopfrinow}.}, connected smooth manifold endowed with a $C^1$ Riemannian metric $g$. In the recent article \cite{mondino2024equivalence}, the equivalence between a distributional Ricci curvature bound $\Ric \geq 0$ and the metric measure space $(M, d_g, \Vol)$ being an $\mathsf{RCD}(0,\dim M)$-space is established. Thus, due to \cite{gigli2013splitting, gigli2014splitoverview}, we know the $\mathsf{RCD}$-splitting theorem is applicable to $(M, d_g, \Vol)$ given the existence of a line. That general result, however, yields a splitting map which is a priori only Lipschitz. We want to show that for $g\in C^1$ it is in fact a $C^2$-isometry. We will also briefly address the case of higher regularity metrics and how the regularity of the splitting map increases in those cases.

We will now present a $C^1$-version of the Cheeger-Gromoll splitting theorem with a proof that does not rely on the calculus of $\mathsf{RCD}$-spaces. Instead, we will rely on approximation arguments via convolution.

\begin{Theorem}[Cheeger-Gromoll splitting theorem for $C^1$-metrics]\label{thm:SplittingTheorem}
    Let $(M,g)$ be a complete, connected Riemannian manifold with metric of regularity $C^1$. Let $\Ric\geq 0$ in the distributional sense, i.e.\ $\Ric(X,X) \geq 0$ in $\mathcal{D}'(M)$ for every $X \in C^{\infty}_c(TM)$. If $(M,g)$ contains a line $c:\R \to M$ (i.e.\ a geodesic satisfying $d_g(c(s),c(t)) = |s-t|$ for every $s,t \in \R$), then it is $C^2$-isometric to a product
    \begin{align*}
        (M,g) \cong (M'\times \R, g'\oplus \de t^2),
    \end{align*}
    where $(M',g')$ is a smooth manifold with a $C^1$-Riemannian metric $g'$, dimension\footnote{In the case $\dim M = 1$, $M'$ is a point and $M = c(\R)\cong \R$.} $\dim M -1$ and $\Ric_{M'}\geq 0$ in the distributional sense.
\end{Theorem}
Let us fix a sequence $\eps_i \searrow 0$. We will denote $g_i = g_{\eps_i}$ (see Remark \ref{Remark: Approximatingmetrics}) and all the related objects will have the subscript index $i$.

As is customary in the context of splitting theorems, we start by considering the approximate Busemann functions $b_t^{\pm}(x):= t - d(x,c(\pm t))$. It is easily seen that for fixed $x \in M$, $t \mapsto b_t^{\pm}(x)$ is nondecreasing, $1$-Lipschitz and bounded above by $d(x,c(0))$. Thus the limits $b^{\pm}(x):= \lim_{t \to \infty} b^{\pm}_t(x)$ exist, are $1$-Lipschitz, and the convergence is locally uniform by Dini's theorem.

\begin{Proposition}\label{prop:busemannSubharmonic}
    The Busemann functions $b^\pm$ are subharmonic in the distributional sense, i.e. for all non-negative $\vphi\in C^\infty_c(M)$ it holds that
    \begin{equation}
        \int_M b^\pm \Delta \vphi \,\de\Vol \geq 0.\label{eq:busemann_subharmonic}
    \end{equation}
\end{Proposition}
\begin{proof}
    Let us  denote $d_t:=d(\,\cdot\, ,c(t))$ and $d_{i,t}:=d_{i}(\,\cdot\, ,c(t))$. Given $\vphi\in C^\infty_c(M)$, let $K \Subset M$ be a compact set with $\supp(\phi) \subseteq K$. Then $d_{i,t} \to d_t$ uniformly on $K$ as $i\to \infty$. Moreover, due to $g_i \to g$ in $C^1_{loc}$ and $\Ric \geq 0$, there exists a sequence $\delta_i\searrow 0$ such that $\Ric_i \geq - (n-1) \delta_i g_i$ on $TM|_K$ (cf.\ Proposition \ref{prop:LpRicConvergence}). Thus 
    \begin{align*}
        \int_M d_{i,t} \Delta_i \phi \, \de\Vol_i = \int_M (\Delta_{i} d_{i,t}) \phi \, \de\Vol_i \leq (n-1) \int_M \sqrt{\delta_i} \coth(\sqrt{\delta_i} d_{i,t}) \phi \, \de\Vol_i,
    \end{align*} 
    because integration happens on the compact set $K$ and standard Laplacian comparison (see e.g.\ \cite[Lem.\ 7.1.9]{petersen2006riemannian}) 
    works there due to the estimate on $K$ for $\Ric_i$. Also
    \begin{align*}
        \int_M d_{i,t} \Delta_i \phi \, \de\Vol_i \longrightarrow \int_M d_t \Delta \phi \,\de\Vol.
    \end{align*}
    Noting that $\lim_{i\to\infty}\sqrt{\delta_i}\coth(\sqrt{\delta_i} d_{i,t}) = \frac{1}{d_{t}}$ uniformly on $K$, we conclude that
    \begin{align*}
    \int_M d_t \Delta \phi \, \de\Vol &= \lim_{i\to\infty} \int_M d_{i,t} \Delta_i \phi \, \de\Vol_i \\
    &\leq \lim_{i\to\infty} (n-1) \int_M \sqrt{\delta_i} \coth(\sqrt{\delta_i} d_{i,t}) \phi \, \de\Vol_i
    = \int_M \frac{n-1}{d_t} \phi \, \de\Vol.    
    \end{align*}
    Note that for any $t \in \R$, $\int_M t \Delta \phi \, \de\Vol = \int_M (\Delta t) \phi \, \de\Vol = 0$, so
    \begin{align*}
        \int_M (t-d_t) \Delta \phi \, \de\Vol \geq - \int_M \frac{n-1}{d_t} \phi \, \de\Vol.
    \end{align*}
    The LHS converges to $\int_M b \Delta \phi \, \de\Vol$ and the RHS to $0$ for $t \to \infty$, so we get that $b^+$ is subharmonic. An analogous proof establishes the same property for $b^-$.
\end{proof}
\begin{Remark}\label{rk:distribLaplacian}
    Let us briefly note that $f\in C^{0,1}(M)$ being \emph{harmonic (subharmonic, superharmonic) in the distributional sense}, i.e.
    \begin{equation*}
        \int_M f \Delta \vphi \,\de\Vol =0 \quad(\text{resp. }\geq 0,\, \leq 0) \qquad \text{for all non-negative }\vphi\in C^\infty_c(M),
    \end{equation*}
    is equivalent to
    \begin{equation*}
        \Delta (f\circ \psi^{-1}) = \frac{1}{\sqrt{\det g}}\partial_\alpha ( \sqrt{\det g} \,g^{\alpha\beta} \partial_\beta( {f\circ\psi^{-1}})) =0 \quad(\text{resp. }\geq 0,\,\leq 0)  \qquad \text{in }\D'(\psi(U))
    \end{equation*}
    for every chart $(U,\psi)$. Indeed, let us prove the subharmonic case, the other cases are analogous. Let $\vphi\in C^\infty_c(M)$ be non-negative and, without loss of generality, we can assume $\supp \vphi$ is fully contained in $U$. Let us denote $\Bar{f}:=f\circ\psi^{-1},\,\Bar{\vphi}:=\vphi \circ \psi^{-1}$. Then we have
    \begin{align*}
        \int_M f \Delta \vphi \,\de\Vol 
        &= \int_{\psi(U)} \Bar{f} \partial_\beta ( \sqrt{\det g} \,g^{\alpha\beta} \partial_\alpha \Bar\vphi)\,\de x 
        = - \int_{\psi(U)}\partial_\beta \Bar{f} \sqrt{\det g} \,g^{\alpha\beta} \partial_\alpha \Bar\vphi \,\de x\\
        &= \langle \partial_\alpha ( \sqrt{\det g} \,g^{\alpha\beta} \partial_\beta \Bar{f}) , \Bar\vphi\rangle \geq 0.
    \end{align*}
    By arbitrariness of $\Bar{\phi}\ge 0$ in $C^\infty_c(\psi(U))$ we have that $\partial_\alpha ( \sqrt{\det g} \,g^{\alpha\beta} \partial_\beta \Bar{f})\geq 0$. This implies that $\partial_\alpha ( \sqrt{\det g} \,g^{\alpha\beta} \partial_\beta \Bar{f})$ is a distribution of order 0, so its product with the positive continuous function ${1}/{\sqrt{\det g}}$ is well-defined and it preserves the non-negativity.
\end{Remark}
\begin{Proposition}
    We have $b^+ = -b^-$. As a consequence, $b^+$ is harmonic in the distributional sense.
\end{Proposition}
\begin{proof}
    Observe that for all $x \in M$, $s \in \R$,
    \begin{align*}
        b^+(x)+b^-(x) &= \lim_{t\to\infty}(2t -d(x,c(t))-d(x,c(-t)))\leq \lim_{t\to\infty}(2t -d(c(t),c(-t)))=0;\\
        b^+(c(s))+b^-(c(s))&= \lim_{t\to\infty}2t -d(c(s),c(t))-d(c(s),c(-t))=0.
    \end{align*}
    Since $b^\pm \in C^{0,1}(M)\subseteq W^{1,2}_\loc(M)$, $\Delta b^\pm \geq 0$ holds in the weak sense as well. Thus by the strong maximum principle for elliptic PDE (see \cite[Thm.\ 8.19]{giltrud2001elliptic}) we can conclude that $b^+ +b^-=0$. Since $b^+=-b^-$, it is both subharmonic and superharmonic, thus harmonic.
\end{proof}
From now on let us denote $b:=b^+$.
\begin{Corollary}
    It holds that $b\in W^{2,\infty^-}_\loc(M)$. Furthermore, $b\in C^{1,\alpha}_\loc(M)$ for all $\alpha\in(0,1)$.
\end{Corollary}
\begin{proof}
    Let us restrict ourselves to a relatively compact open chart domain $\Omega$. As a consequence of Proposition \ref{prop:busemannSubharmonic}, it holds that $b\in W^{1,2}(\Omega)$ is a weak solution of
    \begin{equation*}
        \partial_\alpha \left( \sqrt{\det g}\, g^{\alpha\beta}\partial_\beta u \right)=0 \qquad \text{on } \Omega.
    \end{equation*}
    Due to elliptic regularity (see e.g.\ \cite[Thm.\ 4.9]{giaqmar2012elliptic}), it holds that $b\in W^{2,2}_{\mathrm{loc}}(\Omega)$. Now observe that on $\Omega$, $b$ it is a strong solution of
    \begin{equation*}
        g^{\alpha\beta}\partial_\alpha\partial_\beta b= -\left(\partial_\alpha g^{\alpha\beta} \right)\partial_\beta b - g^{\alpha\beta} \partial_\alpha (\log \sqrt{\det g})\, \partial_\beta b.
    \end{equation*}
    Recall that $b$ is Lipschitz, thus its first order derivatives $\partial_\beta b$ are in $L^p_{\mathrm{loc}}(\Omega)$. Due to $L^p$-estimates for solutions of elliptic PDE \cite[Thm.\ 7.3]{giaqmar2012elliptic}, we get $b\in W^{2,p}_\text{loc}(\Omega)$ for all $p\in [1,\infty)$. 
    The second statement follows directly from the embeddings of Sobolev spaces in H\"older spaces.
\end{proof}

Before proceeding with the next part of the proof, we need to briefly discuss the notion of geodesics when the metric $g$ is just continuously differentiable. Due to the fact that now the Christoffel symbols $\Gamma_{\alpha\beta}^\delta$ are merely continuous, the Cauchy problem
\begin{equation*}
    \begin{cases}
        \Ddot{\gamma}^\delta+ \Gamma_{\alpha\beta}^\delta\Dot\gamma^\alpha\Dot\gamma^\beta=0 \quad \text{for }\delta=1,...,\dim M\\
        \gamma(0)=p, \quad \Dot\gamma(0)=v
    \end{cases}
\end{equation*}
admits at least one solution, due to Peano's Theorem \cite[Ch. 2, Thm. 2.1]{hartman2002ODE}, but it may not be unique.\\
Under the assumption that $(M,d_g)$ is a complete metric space, \cite[Rk.\ 2.5.29]{bubuiv2001metric} ensures that any two given points admit a minimizing curve joining them. As seen in \cite{samstein2018geodlowreg}, under the assumption that $g\in C^1(\cT^0_2 M)$ these curves are $C^2$ and satisfy the geodesic equation.\\
When $g$ is continuously differentiable we also have the following version of the Hopf-Rinow theorem (where by \emph{geodesic} we mean any solution of the aforementioned Cauchy problem):
\begin{Theorem}[$C^1$-Hopf-Rinow]\label{thm:C1hopfrinow}
     Let $(M,g)$ be a connected Riemannian manifold with $g\in C^1(\cT^0_2 M)$. Then the following are equivalent:
     \begin{enumerate}
         \item $(M,d_g)$ is a complete metric space.
         \item $(M,d_g)$ is a proper metric space, i.e.\ closed bounded sets are compact.
         \item $(M,g)$ is geodesically complete, i.e.\ every geodesic $\gamma:(a,b)\to M$ can be extended to a geodesic $\bar{\gamma}:\R\to M$.
     \end{enumerate}
 \end{Theorem}
 \begin{proof}
     The equivalence of the first two conditions is ensured by the Hopf-Rinow-Cohn-Vossen Theorem \cite[Thm.\ 2.5.28]{bubuiv2001metric} (in particular this equivalence holds even for continuous metrics, cf.\ \cite{Burtscher_length}). Due to the same theorem, we also have that the first two conditions are equivalent to
     \begin{equation*}
         \text{all minimizing geodesics }\sigma:[a,b)\to M\text{ admit a continuous extension }\bar{\sigma}:[a,b]\to M,
     \end{equation*}
     which follows from (iii).\\
     Finally let us prove that (i),(ii) $\Rightarrow$ (iii). As for the smooth case, for any geodesic $\gamma$ it holds that
     \begin{align*}
        \dv{}{t}g(\Dot{\gamma},\Dot{\gamma}) = 2 g ({\Ddot\gamma},\Dot{\gamma})=0.
     \end{align*}
     Indeed, following a classical proof such as \cite[Thm. 5.1.2]{petersen2006riemannian} we notice that $\gamma\in C^2((a,b),M)$ and $g\in C^1(\cT^0_2 M)$ are sufficient conditions for the above formula. It follows that integral curves of the corresponding geodesic spray on $TM$ remain in compact sets on finite time intervals, implying completeness of the spray and thereby
     geodesic completeness by projecting to $M$ (cf.\ \cite[Ch. 4, Th.\ 2.3]{Lang}).
 \end{proof}

\begin{Proposition}
    The norm of the gradient satisfies $\abs{\nabla b}=1$ on all of $M$.
\end{Proposition}
\begin{proof}
    Since $b$ is 1-Lipschitz, it is enough to prove that $\abs{\nabla b(x)}\geq 1$ for any fixed $x\in M$. As above, consider the approximate Busemann function $b_t(x)=t-d(x,c(t))$. Given 
    $ t\in\R$ (with $x\neq c(t)$), call $v_t:=\sigma_t'(0)\in T_x M$, where $\sigma_t$ is a length minimizing unit speed curve from $x$ to $c(t)$. Due to Theorem \ref{thm:C1hopfrinow}, 
    each $\sigma_t$ can be extended to all of $\R$ as a solution of the geodesic equation. 
    Setting $d:=d(x,c(t))$ we have $\sigma_t: [0,d]\to M$ and $b_t(\sigma_t(s)) = t+s-d$. Hence $(b_t\circ \sigma_t)'(s) = 1$ for all $s\in [0,d]$.

    Due to compactness of the unit sphere in $T_xM$, we can extract a sequence $t_j\to \infty$ and a unit vector $v\in T_xM$ such that $v_{t_j} \to v$. At the same time, due to the Riemannian analogue of \cite[Cor.\ 2.6]{KOSS22C1HawkingPenrose} applied to the family $\{ \sigma_{t_j}\}_j$, 
    there exists $\delta>0$ and a geodesic $\sigma:[-\delta,\delta]\to M$  such that $\sigma(0)=x, \sigma'(0)=v$ and, up to passing to a subsequence of $\{t_j\}_j$, $\sigma_{t_j} \to \sigma$ in $C^2([-\delta, \delta])$.
    We then obtain that $b_{t_j}\circ \sigma_{t_j} \to b \circ \sigma$ uniformly on $[0, \delta]$. At the same time $(b_{t_j}\circ\sigma_{t_j})'\vert_{[0,\delta]}=1$ identically. Consequently,  $b\circ \sigma \vert_{[0,\delta]}$ is differentiable with $(b\circ \sigma)' \vert_{[0,\delta]}\equiv 1$. Finally, we deduce
    \begin{equation*}
        g(\nabla b(x),v) 
         = (b\circ \sigma)'(0)
        =1.
    \end{equation*}
\end{proof}

Since $b\in W^{2,\infty^-}_\loc(M)$, we can define its Hessian $\Hess b = \nabla \de b\in L^{\infty^-}_\loc(\cT^0_2 M)$. In local coordinates
\begin{equation*}
    \Hess b = \left( \partial_\alpha \partial_\beta b - \Gamma_{\alpha \beta}^\gamma \partial_\gamma b \right) d x^\alpha \otimes d x^\beta.
\end{equation*}
\begin{Proposition}
    $\Hess b =0$ almost everywhere.
\end{Proposition}
\begin{proof}
    Since $\de b \in W^{1,\infty^-}_\loc(T^*M)$, then also $\nabla b \in W^{1,\infty^-}_\loc(TM)$. We can thus apply the Bochner-Weitzenb\"ock identity (Theorem \ref{thm:BochWeitz}). Noting that 
    \begin{align*}
     \Delta_H(\nabla b) = (\Delta_H(db))^\sharp = (d\delta d b)^\sharp = -\nabla (\Delta b),
    \end{align*}
    this gives
    \begin{equation*}
        \Ric(\nabla b,\nabla b)= \Delta \frac{\abs{\nabla b}^2}{2} - \abs{\Hess b}^2 -
            g( \nabla \Delta b, \nabla b)= - \abs{\Hess b}^2.
    \end{equation*}
    Since $\Ric(\nabla b,\nabla b)\geq 0$ (even though $\nabla b$ is not smooth and compactly supported, this is easily seen to hold by approximation), we conclude that $\Hess b =0$.
\end{proof}

\begin{Corollary}
    It holds that $b\in C^2(M)$.
\end{Corollary}
\begin{proof}
    In local coordinates we have that
    \begin{align*}
        \Hess b = \left( \partial_\alpha \partial_\beta b - \Gamma_{\alpha \beta}^\gamma \partial_\gamma b \right) d x^\alpha \otimes d x^\beta=0.
    \end{align*}
    Thus for every $\alpha, \beta$
    \begin{align*}
         \partial_\alpha \partial_\beta b = \Gamma_{\alpha \beta}^\gamma \partial_\gamma b \in C^0 \quad \text{a.e.},
    \end{align*}
so the claim follows from \cite[Ch.\ 4, Prop.\ 10]{Horvath}. 
\end{proof}
After these preparations, we can now establish the desired isometry. 

\begin{Proposition}\label{prop:phiisometry}
    Let $M':= b^{-1}(\{0\})$. Then $M'$ is an embedded $C^2$-hypersurface of $M$ with vanishing second fundamental form and
    can be endowed with a $C^\infty$-structure such that
        \begin{alignat*}{3}
            \Phi: \quad\R &\times M' & \quad&\longrightarrow & &\;\;M\\
            (&t,p)& &\longmapsto &\quad &\Fl_t^{\nabla b}(p)
        \end{alignat*}
        is an isometric $C^2$-diffeomorphism. 
\end{Proposition}
\begin{proof}
    Since $\Hess b=0$, also $\nabla \nabla b=0$ and thus the integral curves of $\nabla b$ are geodesics. By Theorem \ref{thm:C1hopfrinow}, all integral curves of $\nabla b$ are defined on all of $\R$. Since $b \in C^2(M)$ and $\abs{\nabla b}=1$, all level sets $b^{-1}(t)$ are embedded $C^2$-hypersurfaces.\\
    Recall that for $(t,p)\in \R\times M'$, the flow $\Fl_t^{\nabla b} (p)$ of the vector field $\nabla b$ denotes the point $\gamma(t)\in M$, where $\gamma$ is the solution of
    \begin{equation*}
        \begin{cases}
            \Dot\gamma(t) = \nabla b(\gamma(t)),\\
            \gamma(0)=p.
        \end{cases}
    \end{equation*}
    It is a well-known consequence of the flow box theorem (\cite[Prop.\ 4.1.13]{Abraham_Marsden_Ratiu}) that $(t,p)\mapsto \Phi(t,p):= \Fl^{\nabla b}_t(p)$ is a $C^1$-map $\Phi:\R \times M' \to M$, while $t \mapsto \Fl^{\nabla b}_t(p)$ is $C^2$ for every fixed $p \in M$.\\
    Note that for $t \in \R$ and $p \in M'$
        \begin{equation*}
            b(\Fl_t^{\nabla b} (p) )=t,
        \end{equation*}and consequently
        \begin{equation*}
            \Fl_t^{\nabla b} (M' )= b^{-1}(t).
        \end{equation*}
    This immediately implies that $\Phi$ is a $C^1$-diffeomorphism. Let us first check that $\Phi$ is a $C^1$-isometry. Observe that
    \begin{equation}\label{eq:from_here}
       T_{(t,p)}\Phi(a,0) = a \nabla b(\Phi (t,p)).
    \end{equation}
    Fix $p\in M'$. Since $\nabla b(\Phi(t,p)) \perp T_{\Phi(t,p)} (\{b=t\})$ for each $t$, for any $(a,v)\in \R \times T_{p} M$ we have 
    \[g(T_{(t,p)}\Phi(a,v), T_{(t,p)}\Phi(a,v)) = a^2 |\nabla b(t)|^2 + |T_{(t,p)}\Phi(0,v)|^2
    = a^2 + |T_{(t,p)}\Phi(0,v)|^2.
    \] 

    Thus we may conclude the proof by showing that $\abs{T_{(t,p)}\Phi(0,v)}^2$ is independent of $t$ for any given $v\in T_{p}M'$ (because for $t=0$ it holds that $T_{(0,p)} \Phi (0,v) = v$). 
    On $\R\times M'$ let us index with latin letters the coordinates on $\{0\}\times M$ to distinguish them from the flow parameter $t$. Greek letters will be reserved to coordinates on $M$. Locally, by definition, we have
    \begin{align*}
        \partial_t \Phi^\beta (t,p) = (\nabla b)^\beta (\Phi(t,p))
    \end{align*}
    and by symmetry of the second distributional derivatives
    \begin{align*}
        \partial_t\partial_a \Phi^\beta (t,p)=\partial_a\partial_t \Phi^\beta (t,p)= \partial_a[(\nabla b)^\beta (\Phi(t,p))] = \partial_a \Phi^\mu(t,p) \partial_\mu(\nabla b)^\beta (\Phi(t,p)).
    \end{align*}
    Given that these mixed second derivatives are distributions represented by continuous functions, by \cite[Th.\  2.1.2]{Friedlander} the derivatives exist also in the classical sense. Now the following computation is admissible due to the fact that the products are merely multiplications between functions.
    \begin{align*}
        \dv{}{t} \abs{T_{(t,p)}\Phi(0,v)}^2&=\partial_t( g_{\alpha\beta}(\Phi(t,\,.\,))\partial_j\Phi^\alpha v^j\partial_k\Phi^\beta v^k) \\
        &=\partial_t\Phi^\mu \partial_\mu g_{\alpha\beta}\partial_j\Phi^\alpha v^j\partial_k\Phi^\beta v^k +
        2g_{\alpha\beta}\partial_t \partial_j\Phi^\alpha v^j\partial_k\Phi^\beta v^k\\
        &=(\nabla b)^\mu (\Gamma^\sigma_{\mu\alpha}g_{\sigma\beta} + \Gamma^\lambda_{\mu\beta}g_{\alpha\lambda})\partial_j\Phi^\alpha v^j\partial_k\Phi^\beta v^k +
        2g_{\alpha\beta}\partial_j \Phi^\nu \partial_\nu(\nabla b)^\alpha  v^j\partial_k\Phi^\beta v^k\\
        &=2 g(\nabla_{T_{(t,p)}\Phi(0,v)} \nabla b, T_{(t,p)}\Phi(0,v) ) \\
        &= 2(\Hess b)(T_{(t,p)}\Phi(0,v),T_{(t,p)}\Phi(0,v))\\
        &=0.
    \end{align*}
We conclude that $\Phi$ is a $C^1$-isometry. At this point, the regularity of the embedding of $M'$ in $M$ is no longer relevant. We endow $M'$ with a $C^2$-compatible $C^\infty$-structure (cf.\ \cite[Thm.\ 2.9]{Hirsch}). Then $\Phi$ can be understood as a $C^1$-isometry of $C^1$-metrics on smooth manifolds. Thus, a result of Taylor \cite[Thm.\ 2.1]{taylor2006existence} asserting the $C^2$-regularity of distance-preserving maps between $C^1$-Riemannian metrics concludes the proof.
 
\end{proof}
With the following corollary we conclude the proof of Theorem \ref{thm:SplittingTheorem}.
\begin{Corollary}
    In the setting of Theorem \ref{thm:SplittingTheorem}, $M'$ is a $C^\infty$-manifold with a 
    $C^1$-Riemannian metric with $\Ric_{M'} \geq 0$ in the distributional sense.
\end{Corollary}
\begin{proof}
This is immediate, since the distributional Ricci tensor of a product Riemannian manifold splits.
\end{proof}

We saw that if the metric is of regularity $C^1$, then the splitting map is of regularity $C^2$. This continues to hold in higher non-smooth regularities, which we summarize in the following result.
\begin{Theorem}
\label{thm: higherregularitysplitting}
    Under the same hypotheses as Theorem \ref{thm:SplittingTheorem}, assume that $g\in C^k(\cT^0_2 M)$, with $k\in \N$, $k\geq 1$. Then $M'$ is a $C^{k+1}$-embedded hypersurface of $M$ with vanishing second fundamental form. Endowing $M'$ with a 
    $C^{k+1}$-compatible $C^\infty$-structure, the splitting map $\Phi: (\R \times M', dt^2 \oplus g')  \to (M,g)$ defined in Proposition \ref{prop:phiisometry} is an isometry of regularity $C^{k+1}$, and $(M',g')$ is a smooth manifold with Riemannian metric $g' \in C^k(\cT^0_2 M')$ having non-negative Ricci curvature.
\end{Theorem}
\begin{proof}
    Following the proof of Theorem \ref{thm:SplittingTheorem}, recall that $\Hess b =0$ and $b\in C^2(M)$. 
   Thus we proceed by a bootstrapping argument:
    \begin{align*}
        \partial_\alpha \partial_\beta b &= \Gamma_{\alpha \beta}^\gamma \partial_\gamma b \in C^1 &  &\Longrightarrow & b &\in C^3(M), \\
        & \hspace{0.6em}\vdots &  & \hspace{0.8em} \vdots & & \  \hspace{0.8em} \vdots
        \\
        \partial_\alpha \partial_\beta b &= \Gamma_{\alpha \beta}^\gamma \partial_\gamma b \in C^{k-1} &  & \Longrightarrow & b &\in C^{k+1}(M).       
    \end{align*}
    Recall that the splitting map is $\Phi(t,p)= \Fl_t^{\nabla b}(p)$, which is just the point $u(t)$ where $u:\R\to M$ is the solution to
    \begin{align*}
        \begin{cases}
            \Dot{u}(t) = \nabla b (u(t)),\\
            u(0)=p.
        \end{cases}
    \end{align*}
    Due to the flowbox theorem \cite[Thm.\ 4.1.13]{Abraham_Marsden_Ratiu}, the map $\Phi$ is of the same regularity as $\nabla b$, which is of the same regularity as $g$, i.e., $C^k$. Once again, invoking \cite[Thm.\ 2.1]{taylor2006existence} in the same way as in the proof of Proposition \ref{prop:phiisometry} yields $C^{k+1}$-regularity of the isometry, concluding the proof.
\end{proof}

\begin{Remark}[Weighted generalizations]
It is possible to generalize Theorems \ref{thm:SplittingTheorem} and \ref{thm: higherregularitysplitting} to Riemannian manifolds whose volume measure is weighted by $e^{-V}$, where $V \in C^1(M)$. This is currently work in progress in Lorentzian signature \cite{Quintetellipticsplitting24}, where an elliptic approach to Lorentzian splitting theorems \cite{braun2024elliptic} based on the $p$-d'Alembertian (studied in a synthetic setting in \cite{Octet24+}) is used for low regularity metrics. The observations in that work regarding the effects of the weight are not signature-dependent, and so they provide the arguments needed for a weighted generalization of our results, as claimed.
\end{Remark}

\subsection{Flatness for $C^1$ semi-Riemannian manifolds}
\label{Subsection: Flatness}

The purpose of this subsection is to characterize flatness for metrics of regularity $C^1$. Since no properties specific to Riemannian signature are needed for this, we will (for this section only) consider arbitrary (non-degenerate) semi-Riemannian metrics $g$. For smooth metrics, it is well-known that local isometry to flat space is equivalent to the vanishing of the Riemann curvature tensor. This continues to be the case if the metric is $C^1$, as we will show now.

For the remainder of this subsection, we fix a semi-Riemannian manifold $(M,g)$ such that $g$ is of regularity $C^1$ with signature $(k,l)$, $k,l \geq 0$ and $k+l = n := \dim M$. We also fix a background complete (smooth) Riemannian metric $h$ on $M$. Let $g_{\varepsilon}$ be the smooth semi-Riemannian metrics of signature $(k,l)$ from Remark \ref{Remark: Approximatingmetrics}. We first note the following consequence of the vanishing of the distributional Riemann curvature tensor. 

\begin{Proposition}[Consequences of $\Riem_g=0$]
\label{Proposition: SmallnessofapproxRiemann}
Let $(M,g)$ be a $C^1$ semi-Riemannian manifold satisfying $\Riem_g=0$. Then for the Riemann curvature tensors $\Riem_{\varepsilon}$ of the smooth approximating metrics $g_{\varepsilon}$ the following holds: For all compact $K\Subset M$, for all $C, \delta > 0$ there exists $\varepsilon_0 > 0$ such that for all $\varepsilon<\varepsilon_0$ and for all $C^1$-vector fields $X_1,X_2,X_3$ on $K$ with $\abs{X_j}_h \leq C$, we have that $\abs{\Riem_{\varepsilon}(X_1,X_2)X_3}_h \leq \delta$.
\begin{proof}
    By assumption, $\Riem(X_1,X_2)X_3 = 0$, hence also $(\Riem(X_1,X_2)X_3)\star_M \rho_{\varepsilon} = 0$. 
    It follows from \cite[Prop.\ 4.2(ii) and Rem.\ 4.4]{KOV:22} that 
    $(\Riem(X_1,X_2)X_3) \star_M \rho_{\varepsilon} - \Riem_{\varepsilon}(X_1,X_2)X_3 \to 0$ in $C^0_{loc}$, 
    , giving the claim.
\end{proof}
\end{Proposition}

We shall also require the following observation on the dependence of solutions of the parallel transport equation on parameters and initial conditions (cf.\ \cite[Ch.\ V]{hartman2002ODE}): Consider, in some local chart, solutions of the initial value problem for parallel transport of a vector $v$ along a smooth family of curves $(t,x)\mapsto c(t,x)$ ($t\in \R$, $x\in \R^m$):
\begin{equation}\label{eq:eps_parallel_transport}
\begin{split}
    \frac{dV_\eps^k}{dt}(t,x,v) &= - {}^{g_\eps}\Gamma^k_{ij}(c(t,x)) \frac{d(x^j\circ c)}{dt} V_\eps^i(t)(t,x,v), \quad V_\eps(0,x,v) = v \\
    \quad \frac{dV^k}{dt}(t,x,v) &= - {}^{g}\Gamma^k_{ij}(c(t,x)) \frac{d(x^j\circ c)}{dt} V^i(t,x,v), \quad V_\eps(0,x,v) = v 
\end{split}
\end{equation}
Then $V$ is continuously differentiable with respect to $t$, and we have the following convergence properties: $[t\mapsto V_\eps(t,x,v)] \to [t\mapsto V(t,x,v)]$ in $C^1_{\text{loc}}$, locally uniformly in $(x,v)$, and $V_\eps \to V$ locally uniformly in $(t,x,v)$. The same convergence properties hold if on the right hand side of the $V_\eps$-equation a term $b_\eps(x)V_\eps(t,x,v)$ is added, with $b_\eps$  smooth in $x$ and converging to $0$
locally uniformly as $\eps\to 0$.

The following Proposition contains the main arguments needed to prove the flatness criterion.

\begin{Proposition}[Extending vectors to parallel vector fields]
\label{Proposition: Extendingtoparallelvf}
Let $(M,g)$ be a $C^1$ semi-Riemannian manifold satisfying $\Riem_g = 0$. Then for any $p \in M$ and $v \in T_pM$, there exists an open neighborhood $U$ around $p$ and a parallel vector field $V \in C^1(TU)$ with $V_p = v$.
\begin{proof}
As this is a local question, suppose $M=\R^n$ and $p=0$. Consider an open cube $\{x \in \R^n \mid |x^j| < \eta \quad \forall j\} =: C_{\eta}$ around $0$.\par 
We start with $v \in T_pM$ and parallel transport it along the $x^1$-axis (inside $C_{\eta}$) with respect to $g_{\varepsilon}$, then from each point on the $x^1$-axis we parallel transport the resulting vector field along the $x^2$-axis (again with respect to $g_{\varepsilon}$) and so on. By smooth dependence on initial data and the right hand side in the parallel transport equation, this construction yields a smooth vector field $V_{\varepsilon}$ on $C_{\eta}$. By construction, $\nabla^{\varepsilon}_{\partial_1} V_{\varepsilon} = 0$ along the $x^1$-axis in $C_{\eta}$, $\nabla^{\varepsilon}_{\partial_2} V_{\varepsilon} = 0$ in the $x^1x^2$-plane in $C_{\eta}$, and so on (where $\nabla^{\varepsilon}$ denotes the Levi-Civita connection with respect to $g_{\varepsilon}$). Generally, on $M_k:=\{x^{k+1} = \dots = x^n = 0\} \cap C_{\eta}$, we have $\nabla^{\varepsilon}_{\partial_k} V_{\varepsilon} = 0$. Let us now determine, e.g., $\nabla^{\varepsilon}_{\partial_1} V_{\varepsilon}$ in $M_2$. For this, we calculate
\begin{align*}
    \nabla^{\varepsilon}_{\partial_2}\nabla^{\varepsilon}_{\partial_1} V_{\varepsilon} = \nabla^{\varepsilon}_{\partial_1}\nabla^{\varepsilon}_{\partial_2} V_{\varepsilon} + \Riem_{\varepsilon}(\partial_2,\partial_1)V_{\varepsilon}.
\end{align*}
The first term vanishes in $M_2$, so we are left with estimating the curvature term. We take $K:=\overline{C_{\eta/2}}$, then, 
by the above observation
$V_\eps$ converges in $C^0$ to the corresponding solution of the $g$-parallel transport equation. In particular there exists $C >0$ such that on $K$ (with $|\, .\,|$ denoting the Euclidean norm)
\begin{align*}
    |V_{\varepsilon}| \leq C \quad \forall \varepsilon.
\end{align*}
Now let $\delta > 0$ be arbitrary, then by Proposition \ref{Proposition: SmallnessofapproxRiemann} there exists $\varepsilon_0$ such that for all $\varepsilon < \varepsilon_0$ and all $1\le i,j \le n$, we have
\begin{align*}
    |\Riem_{\varepsilon}(\partial_i,\partial_j)V_{\varepsilon}| < \delta \quad \text{on } K.
\end{align*}
This shows that $\nabla^{\varepsilon}_{\partial_1} V_{\varepsilon}$ is almost $g_{\varepsilon}$-parallel along the $x^2$-lines. Now let $W_{\varepsilon}$ be the vector field in $M_2$ that arises by $g_{\varepsilon}$-parallel transporting $\nabla^{\varepsilon}_{\partial_1} V_{\varepsilon}$ from points on the $x^1$-axis along the $x^2$-lines. We will now make precise that $W_{\varepsilon}$ and $\nabla^{\varepsilon}_{\partial_1}V_{\varepsilon}$ are close. Since the difference of $W_{\varepsilon}$ and $\nabla^{\varepsilon}_{\partial_1} V_{\varepsilon}$ is the solution of the parallel transport equation with an additional error term $\Riem_{\varepsilon}(\partial_2,\partial_1)V_{\varepsilon}$ which tends to $0$ uniformly, the difference $W_{\varepsilon} - \nabla^{\varepsilon}_{\partial_1}V_{\varepsilon}$ converges in $C^0$ to the zero vector field.
But then necessarily, $\nabla^{\varepsilon}_{\partial_1}V_{\varepsilon}$ converges to the $g$-parallel transport along the $x^2$-lines of the vectors $\nabla_{\partial_1}^{\varepsilon}V_{\varepsilon}$ evaluated along the $x^1$-axis (because $W_{\varepsilon}$ does so). Moreover, by construction $\nabla^{\varepsilon}_{\partial_1}V_{\varepsilon} = 0$ on the $x^1$-axis, hence $W_{\varepsilon} = 0$, so $\nabla^{\varepsilon}_{\partial_1}V_{\varepsilon} \to 0$ in $M_2$. Iterating this procedure for all derivatives in all coordinate directions, one sees that $\nabla_{\partial_k}^{\varepsilon} V_{\varepsilon}$ is uniformly small on $C_{\eta/2}$. 
By construction, $V_{\varepsilon} \to V$ in $C^0$, where $V$ is the $C^1$-vector field obtained from $g$-parallel transporting $v$ along $x^1$, then along $x^2$ and so on (note that the $V$ so constructed is indeed $C^1$ by \eqref{eq:eps_parallel_transport}). 
But then, again by \eqref{eq:eps_parallel_transport}, also
$\nabla^{\varepsilon}_{\partial_k}V_{\varepsilon} \to \nabla_{\partial_k} V$, and we obtain that $V$ is $g$-parallel in $C_{\eta/2}$, concluding the proof.
\end{proof}
\end{Proposition}

\begin{Lemma}[Coordinates from $C^1$-vector fields]
\label{Lemma: Coordinatesfromvectorfields}
Let $N$ be a smooth manifold of dimension $n$ and let $E_1,\dots,E_n$ be $C^1$-vector fields defined on an open set $U \subseteq N$ satisfying $[E_i,E_j] = 0$ for all $i,j$. Then there exist $C^2$-coordinates $\phi=(x^j):U \to \phi(U) \subseteq \R^n$ onto an open set in $\R^n$ such that $\frac{\partial}{\partial x^j} = E_j$.
\begin{proof}
This is proven in the same manner as the smooth case, for the latter see e.g.\ \cite[Thm.\ 9.46]{LeeSmoothmanifolds2nded}.
\end{proof}
\end{Lemma}

Let us now come to the main result:

\begin{Theorem}[Flatness for $C^1$-metrics]
\label{Theorem: FlatnesscriterionC1}
Let $(M,g)$ be an $n$-dimensional $C^1$ semi-Riemannian manifold of signature $(k,l)$. Then the following are equivalent:
\begin{enumerate}
    \item $(M,g)$ is \textit{(rank one) distributionally curvature-flat}, i.e.\ $\Riem_g = 0 \in \mathcal{D}^{'(1)}(T^1_3 M)$.
    \item $(M,g)$ is \textit{$C^1$ frame-flat}, i.e.\ for each $p \in M$ there exists a neighborhood $U$ of $p$ and an orthonormal frame $E_1,\dots,E_n$ on $U$ of regularity $C^1$ consisting of parallel vector fields.
    \item $(M,g)$ is \textit{$C^2$ coordinate-flat}, i.e.\ for each $p \in M$ there is a coordinate neighborhood $(U,\phi)$ containing $p$ such that the coordinate map $\phi: U \to \phi(U) \subseteq \R^n$ is a $C^2$-diffeomorphism from $U$ onto an open subset of $\R^n$ satisfying $\phi^* g_{k,l} = g$, where $g_{k,l}$ is the semi-Euclidean metric of signature $(k,l)$.
\end{enumerate}
\begin{proof}
    $(i) \Rightarrow (ii)$: Suppose that $\Riem_g=0$. Let $p \in M$ and let $(b_1,\dots,b_n)$ be a $g$-orthonormal basis of $T_pM$. By Proposition \ref{Proposition: Extendingtoparallelvf}, we can extend the $b_i$ to parallel $C^1$ vector fields $E_i$ in a neighborhood $U$ around $p$. Since the $E_i$ are parallel, they constitute an orthonormal frame in $U$.\par
    $(ii) \Rightarrow (iii)$: Let $E_1,\dots,E_n$ be a parallel orthonormal frame of regularity $C^1$ defined on an open set $U$. Then their Lie brackets vanish:
    \begin{align*}
        [E_i,E_j] = \nabla_{E_i}E_j - \nabla_{E_j} E_i = 0.
    \end{align*}
    Thus, by Lemma \ref{Lemma: Coordinatesfromvectorfields}, there exist coordinates $x^j$ on $U$ such that $\frac{\partial}{\partial x^j} = E_j$, which are the required flat $C^2$-coordinates.\par 
    $(iii) \Rightarrow (i)$: Given $C^2$-flat coordinates, it is a purely symbolic calculation (just like in the smooth setting) to prove that $\Riem_g$ is the pullback of the vanishing curvature tensor on $\R^{k,l}$, hence it is also zero.
\end{proof}
\end{Theorem}

\begin{Corollary}[Flatness for $C^k$-metrics]
\label{Corollary: FlatnessforCkmetrics}
If $g \in C^k$ in the setting of Theorem \ref{Theorem: FlatnesscriterionC1}, with $\mathbb{N} \ni k \geq 2$, then, in the equivalent descriptions of flatness, the Riemann curvature tensor is $C^{k-2}$, the local parallel orthonormal frame $(E_1,\dots,E_n)$ is $C^k$, and the flat coordinates are $C^{k+1}$. 
\end{Corollary}

Thus, we have shown that the usual notions of flatness (frame-flat, coordinate-flat, and curvature-flat) all agree for semi-Riemannian metrics of regularity $C^1$. We note that a variant of these results was established previously in Riemannian signature by Taylor \cite[Thm.\ 3.1]{taylor2006existence} using elliptic methods, which do not readily generalize to arbitrary semi-Riemannian signature.

Such questions appear in various low regularity rigidity problems, see e.g.\ the rigidity of the low regularity positive mass theorem in \cite{LeeLefloch2015Positivemasstheorem}. Below $C^1$-regularity, significant difficulties arise. E.g. for Lipschitz frames with vanishing Lie brackets (in a suitable sense), the integrated coordinates are usually only Lipschitz \cite{RampazzoSussmann2007Commutators}. We relegate to a future project the task of adapting the convolution and approximation based techniques to Lipschitz metrics in order to prove a variant of the flatness criterion in that regularity.

\section[Synthetic and distributional curvature bounds]{Synthetic and distributional curvature bounds for non-smooth Riemannian manifolds}
\label{Section: Syntheticdistributional}
When considering smooth manifolds endowed with a Riemannian metric of low regularity (usually less than $C^2$), two main notions of Ricci curvature bounds from below can be applied: the \emph{distributional} bound, i.e.\ $\Ric(X,X) \geq K$ in the sense of distributions for every $X \in C^{\infty}_c(TM)$, as seen in the previous sections of this paper, and the \emph{synthetic} $\mathsf{RCD}(K,N)$ condition on the metric measure space $(M,d,\Vol)$, with $N\geq \dim M$ (see \cite{ambrosiogiglisavare2014metric, erkust2015equivalence}).
This raises the natural question whether 
the two conditions are equivalent (and, if so, down to which regularity of $g$). First progress in this direction was provided by \cite{KOV:22}, while
more recently Mondino and Ryborz proved in \cite{mondino2024equivalence} the equivalence in the weighted case for $g\in W^{1,2}_\loc(\cT^0_2 M) \cap C^0(\cT^0_2 M)$, up to a volume growth bound.

We will now present our own version of said results, developed independently from \cite{mondino2024equivalence}. The second implication to be presented below is a slight improvement in a special case of \cite[Thm. 7.1]{mondino2024equivalence}, as it does not require the volume growth bound and we explicitly construct a pointed measured Gromov-Hausdorff approximating sequence of smooth Riemannian manifolds, while the first one is strictly weaker than \cite[Thm. 6.11]{mondino2024equivalence}. We will nevertheless present our proof, as our stronger hypotheses allow for more direct and transparent arguments.

\subsection{Compatibility between different notions of Sobolev spaces}
\label{Subsection: CompatibilitybetweennotionsofSobolev}
Unless otherwise stated, let $(M,g)$ be a (metrically) complete, connected Riemannian manifold with $C^{0,1}_\loc$ Riemannian metric $g$. Assume that the metric measure space $(M,d,\Vol)$ satisfies the $\mathsf{RCD}(K,\infty)$ condition for some $K\in \R$. We choose to assume initially the $\mathsf{RCD}(K,\infty)$ condition as it is a priori weaker than $\mathsf{RCD}(K,\dim M)$. Observe, however, that Theorem \ref{thm:RCDtodistributional} from the next section, together with Theorem \ref{thm:distribtosynth}, proves that the two conditions are equivalent in this setting. For $C^\infty$ Riemannian manifolds, this is well-known ( \cite[Cor.\ 2.6]{han2020measure}).

We shall assume from now on familiarity with the second order calculus on $\mathsf{RCD}(K,\infty)$ spaces developed by Gigli in \cite{gigli2018nonsmooth}. See also \cite{gigpas2020notes} for a more didactic survey; we will use the latter as main reference. The cited sources rely on a definition of Sobolev spaces, initially developed in \cite{cheeger1999sobolev} and \cite{amgisa2014calculus}, which is tailored towards general metric measure spaces, where a priori we do not have smooth coordinates. We point towards \cite{ambrosio2024metric} for a detailed overview of different (equivalent) methods to define Sobolev spaces on metric measure spaces. We will now explain why these, and the differential operators associated with them, coincide with the notion of Sobolev spaces we used so far.
\begin{Remark}
    We define
    \begin{gather*}
        W^{1,2}(M) := \left\{ f\in W^{1,2}_\loc(M): \int f^2 + \abs{\de f}^2 \,\de\Vol <+\infty\right\},\\
        \norm{f}_{W^{1,2}(M)}^2:= \norm{f}_{L^2(M)}^2 + \norm{\abs{\de f}}_{L^2(M)}^2.
    \end{gather*}
    Let us temporarily denote by $W^{1,2}_*(M)$ the Sobolev space defined on the metric measure space $(M,d,\Vol)$ and by $\de_*$ its associated differential operator. Observe that $C^\infty_c(M)\subseteq W^{1,2}(M)\cap W^{1,2}_*(M)$ because, due to the properties of the classical derivatives, the minimal weak upper gradient $\abs{\de_* f}$ and the norm of the differential $\abs{\de f}$ coincide {for any $f\in C^\infty_c(M)$}; a proof of this fact in the case $M=\R^n$ with the Euclidean metric can be found in \cite[Subsec.\ 2.1.5]{gigpas2020notes}, and it is easily seen to generalize to the class of metrics we consider. This can be extended to $C^{0,1}_c$ functions via mollification. Indeed we know that $C^{0,1}_c(M)\subseteq W^{1,2}(M)\cap W^{1,2}_*(M)$. Also, $f_n :=f \star_M \rho_\frac{1}{n} \rightarrow f$ in $W^{1,2}(M)$. In particular, $\lim_n \abs{\de_* f_n} = \lim_n \abs{\de f_n} = \abs{\de f}$. Thus, up to passing to subsequences, $\de_* f_n$ is a weakly converging sequence of 1-forms and by closure of $\de_*$ \cite[Thm. 4.1.2]{gigpas2020notes} we have that $\de_* f = \lim_n \de_* f_n$ and $\abs{\de_* f}=\lim_n \abs{\de_* f_n} = \abs{\de f}$. We can conclude that $\norm{f}_{W^{1,2}(M)} = \norm{f}_{W^{1,2}_*(M)}$.\\
    Observe that $C^{0,1}_c(M)$ is dense in both spaces: in $W^{1,2}(M)$ due to mollification and in $W^{1,2}_*(M)$ due to \cite[Lem. 6.2.12]{gigpas2020notes}, recalling that bounded supports are compact due to the completeness of $(M,g)$. We can conclude that $W^{1,2}(M) = W^{1,2}_*(M)$ and that the weak upper gradient and the norm of the weak differential coincide. As a byproduct, we also proved that $C^\infty_c(M)$ is dense in $W^{1,2}(M)$.
\end{Remark}
One important consequence of this remark is that due to the uniqueness of their constructions, all the tensor modules defined in \cite{gigpas2020notes} coincide with the spaces of $L^2$-tensor fields defined through coordinates.

We will call \emph{test functions}, \emph{test 1-forms} and \emph{test vector fields} the elements of, respectively,
\begin{align*}
    \Test(M)&:=\left\{f\in C^{0,1}(M)\cap L^\infty(M) \cap W^{1,2}(M):  \Delta f\in L^\infty(M) \cap W^{1,2}(M)  \right\},\\
    \TestF(M)&:=\left\{\sum_{i=1}^k f_i \,\de h_i \in L^2(T^*M): k\in \N,\, f_i,h_i \in \Test(M)  \right\}\\
    \TestV(M)&:=\left\{\sum_{i=1}^k f_i \nabla h_i \in L^2(TM): k\in \N,\, f_i,h_i \in \Test(M)  \right\}.
\end{align*}
Here, $\Delta$, $\de$ and $\nabla$ are understood in the distributional sense. Our definition of test functions is taken from \cite{gigpas2020notes}, but throughout the literature the definition can differ slightly. $\Test$ functions in metric measure spaces play the role that is played by $C^\infty_c$-functions in smooth Riemannian manifolds in the definitions of differential operators through integration by parts. As we lower the regularity of the Riemannian metric $g$, though, it is important to observe that $C^\infty_c$ functions may not be test functions anymore. The following result shows that such a discrepancy does not occur if $g \in C^{0,1}_\loc \cap W^{2,2}_{\loc}$:
\begin{Proposition}\label{lem:temp_1}
    Let $(M,g)$ be a complete, connected Riemannian manifold with $g\in C^{0,1}_\loc(\cT^0_2 M) \cap W^{2,2}_\loc(\cT^0_2 M)$, such that $(M,d,\Vol)$ is an $\mathsf{RCD}(K,\infty)$ space for some $K\in \R$.  Then 
    \begin{equation*}
        C^\infty_c(M)\subseteq \Test(M), \qquad C^\infty_c(T^*M)\subseteq \TestF(M).
    \end{equation*}
    Furthermore, if $g\in C^{1,1}_\loc(\cT^0_2 M) \cap W^{3,2}_\loc(\cT^0_2 M)$ then
    \begin{equation*}
        C^\infty_c(TM)\subseteq \TestV(M).
    \end{equation*}
\end{Proposition}
\begin{proof}
    Let $f\in C^{1,1}_c(M) \cap W^{3,2}_\loc(M)$. Clearly $f$ is bounded and Lipschitz continuous. Also using the Sobolev chain rule  and Leibniz rule for derivations, we have
    \begin{align*}
        \Delta f = \frac{1}{\sqrt{\abs{g} }} \partial_\alpha \left(g^{\alpha\beta} \sqrt{\abs{g}} \partial_\beta f\right) \in W^{1,2}(M)\cap L^\infty(M).
    \end{align*}
    Given that $C^\infty_c(M)\subseteq C^{1,1}_c(M) \cap W^{3,2}_\loc(M)$, we have proven the first inclusion.\\
    Now consider any $\omega\in C^\infty_c(M)$. Then without loss of generality we can assume $\supp \omega$ is contained in a chart and $\omega = \omega_\alpha \de x^\alpha$, with $\omega_\alpha \in C^\infty_c(M)\subseteq \Test (M)$ and, up to multiplication with an appropriate cutoff function, $x^\alpha\in C^\infty_c (M) \subseteq \Test (M)$.\\
    Finally consider any $X\in C^\infty_c(TM)$. The argument is the same as the previous point, decomposing $X$ into $X= X^\alpha g_{\alpha\beta} \nabla x^\beta$. Then $X^\alpha g_{\alpha\beta} \in C^{1,1}_c(M) \cap W^{3,2}_\loc(M)\subseteq \Test(M)$ and, again up to multiplication with a cutoff function, $x^\alpha\in C^\infty_c (M) \subseteq \Test (M)$.
\end{proof}

\begin{Example}\label{Example: Counterexample to Cinfty subset test}
    The inclusion $C^\infty_c(M)\subseteq \Test(M)$ may fail at lower regularities:
    Consider $\R^n$ with the metric $g\in C^{0,1}_\loc(\R^n)\cap W^{2,1}_\loc(\R^n)$ defined via
    \begin{align*}
        g:=\begin{pmatrix}
                \frac{1}{1+\abs{x_1}^{\frac{3}{2}}} &  & & \\
                 & 1 & & \\
                 &  &\ddots & \\
                 & & &1
                \end{pmatrix}.
    \end{align*}
    It is easy to check that $\partial_1\partial_1 {g_{11}}$ behaves asymptotically near $0$ like $|x_1|^{-\frac{1}{2}}$, hence it is not in $L^2_\loc(\R^n)$. Observe that the distributional Riemann curvature tensor vanishes (due to the product structure of $g$), thus in particular $\Ric = 0$. Now take $\varphi \in C^\infty_c(\R^n)$ such that $\varphi= x_1$ locally near $0$. Then after computing $\Delta \varphi$ near $0$ we get
    \begin{align*}
        \Delta \phi &=  \frac{1}{\sqrt{\abs{g}}} \partial_\alpha (\sqrt{\abs{g}} g^{\alpha\beta} \partial_\beta (x_1))
        =  \frac{1}{\sqrt{g_{11}}} \partial_1 (\sqrt{g_{11}} g^{11})\\
        &= \sqrt{g^{11}} \partial_1 \sqrt{g^{11}}
        = \frac{1}{2}\partial_1 g^{11} \notin W^{1,2}_\loc(\R^n).
    \end{align*}
    This shows that $\varphi\notin\Test(M)$.
\end{Example}
Let us go back to the case $g\in C^{0,1}_\loc(\cT^0_2 M)$. Due to low regularity phenomena as illustrated by Example \ref{Example: Counterexample to Cinfty subset test}, the differential operators defined in \cite{gigpas2020notes} are not a priori the same as their distributional counterparts. The exceptions to this are $\Delta$ and $\Div$, as the density of $C^\infty_c(M)$ in $W^{1,2}(M)$ ensures that $\Delta$ and $\Div$ as defined in \cite{gigpas2020notes} are equivalent to the respective distributional operators whenever they take values in $L^2(\meas)$. The reason is that they are defined  via integration by parts with respect to test objects as defined above.
\begin{Proposition}
    Let $\omega \in C^{0,1}_c(T^*M)$. Then its weak exterior differential $\de \omega$ (defined via weak derivatives) is also its exterior differential in the sense of \cite[Def. 6.4.2]{gigpas2020notes}, i.e.
    \begin{equation*}
        \int \de\omega(X_1,X_2)\,\de\Vol = \int \omega(X_1)\Div X_2 - \omega(X_2)\Div X_1 - \omega([X_1,X_2])\,\de\Vol
    \end{equation*}
    for all $X_1,X_2\in \TestV(M)$.
\end{Proposition}
\begin{proof}
    From \cite[Prop. 6.3.7]{gigpas2020notes} we know that in the case of $X_1,X_2\in \TestV(M)$ the vector field $[X_1,X_2]$ is exactly the one associated with the derivation $H^{2,2}(M)\cap C^{0,1}(M)\ni f \mapsto X_1(X_2(f)) - X_2(X_1(f))$. Since $C^\infty_c(M)\subseteq D(\Delta)\subseteq H^{2,2}(M)$, $[X_1,X_2]$ coincides with the usual Lie bracket of vector fields.
    
    Without loss of generality we may assume that the support of $\omega$ is contained in a single chart and that $\omega = f \de x$ for some $f\in C^{0,1}_c(M),\,  x\in C^\infty_c(M)$. Then
    \begin{align*}
        &\int \de\omega(X_1,X_2)\,\de\Vol = \int \de f(X_1) \de x(X_2) -\de f(X_2) \de x(X_1) \,\de\Vol\\
        &= -\int[ \Div (\de x(X_2) X_1) - \Div (\de x(X_1)  X_2 )] f \,\de\Vol.
    \end{align*}
    Observe that in coordinates
    \begin{align*}
        {\Div (\de x(X_1) X_2)}
        &= \frac{1}{\sqrt{\abs{g}}}\partial_\alpha (\sqrt{\abs{g}} \de x (X_1) X_2^\alpha) = \de x (X_1) \Div X_2 + X_2^\alpha \partial_\alpha (\de x (X_1)) \\&= \de x (X_1) \Div X_2 + X_2(X_1(x)).
    \end{align*}
    Similarly ${\Div (\de x(X_2) X_1)}
    =  \de x (X_2) \Div X_1 + X_1(X_2(x))$. Both are well defined $L^1_\loc(M)$ functions.
    Thus we conclude
    \begin{align*}
        &\int \de\omega(X_1,X_2)\,\de\Vol\\
        &= -\int f [ \de x(X_2) \Div X_1 - \de x(X_1) \Div X_2 + X_1(X_2(x)) - X_2(X_1(x))] \,\de\Vol\\
        &=\int \omega(X_1)\Div X_2 - \omega(X_2)\Div X_1 - \omega([X_1,X_2])\,\de\Vol.
    \end{align*}
\end{proof}

\begin{Proposition}
    Let $\omega \in C^{0,1}_c(T^*M)$. Then its weak codifferential $\delta \omega$ is also its codifferential in the sense of \cite[Def. 6.4.16]{gigpas2020notes}, i.e.
    \begin{equation*}
        \int g( \omega,\de f )\,\de\Vol = \int f \delta\omega \,\de\Vol
    \end{equation*}
    for all $f\in \Test(M)$.
\end{Proposition}
\begin{proof}
    Recall that by definition
    \begin{equation*}
        \int g( \omega,\de f )\,\de\Vol = \int f \delta\omega \,\de\Vol
    \end{equation*}
    already holds for all $f\in C^\infty_c(M)$. By density the above formula can be extended to all $f\in W^{1,2}(M)$, thus including $f\in \Test(M)$.
\end{proof}
The two previous propositions can be summarized by saying that $C^{0,1}_c(T^*M)\subseteq W_H^{1,2}(T^*M)$, where
\begin{align*}
    W_H^{1,2}(T^*M)&:=\left\{ \omega\in L^2(T^*M): \omega \text{ admits } \de\omega\in L^2(\cT^0_2 M),\;  \delta \omega \in L^2(M) \text{ in the sense of \cite{gigpas2020notes}}\right\}.
\end{align*}
Since $\TestF(M)\subseteq W^{1,2}_H(T^*M)$, we denote by $H^{1,2}_H(T^*M)$ the closure of $\TestF(M)$ with respect to the norm
\begin{align*}
    \norm{ \omega }^2_{W^{1,2}_H(T^*M)}:=\norm{ \omega }^2_{L^2(T^*M)}+\norm{ \de\omega }^2_{L^2(\cT^0_2 M)}+\norm{ \delta \omega}^2_{L^2(M)}.
\end{align*}
Analogous definitions can be given to vector fields.
\begin{align*}
    W_H^{1,2}(TM)&:=\left\{X\in L^2(TM) : X^\flat \in W^{1,2}_H(T^*M)\right\},\\
    H^{1,2}_H(TM) &:= \left\{X\in L^2(TM) : X^\flat \in H^{1,2}_H(T^*M)\right\}.
\end{align*}
We will require the following lemma on approximations via test functions.
\begin{Lemma}
    Let $f\in C^{0,1}_c(M)$. Then there exists a sequence $\{f_n\}_n$ of $\Test(M)$ functions such that $f_n\to f$ in $W^{1,2}(M)$ with $\norm{f_n}_{L^\infty},\, \norm{\de f_n}_{L^\infty}$ uniformly bounded.\\
    Let $x\in C^\infty_c(M)$. 
    Then there exists a sequence $\{x_n\}_n$ in $\Test(M)$ such that $x_n\to x$ in $W^{1,2}(M)$ and $\Delta x_n \to \Delta x$ in $L^2(M)$.
\end{Lemma}
\begin{proof}
    The first statement has been proven in \cite[Section 3.2]{gigli2018nonsmooth}. The second statement follows by definition upon applying the heat flow to $x$ (\cite[Ch.\ 5]{gigpas2020notes}).
\end{proof}

\begin{Proposition}\label{prop:H12vectors}
    It holds that $C^\infty_c(TM)\subseteq H_H^{1,2}(TM)$.
    
\end{Proposition}
\begin{proof}
    Take $V\in C^\infty_c(TM)$, and without loss of generality suppose that
    $\supp V$ is fully contained in a chart domain. In this case we can treat smooth coordinate functions $x^\alpha$ as $C^\infty_c(M)$ functions, since we can just multiply $x^\alpha$ by $\chi\in C^\infty_c(M)$ such that $\chi=1$ on $\supp V$ and $\supp \chi$ is fully contained in said chart. Let $\omega= V^\flat\in C^{0,1}_c(T^*M)$; up to linearity we can write $\omega = f\de x$ where $f\in C^{0,1}_c(M)$ and $x\in C^\infty_c(M)$. Now consider $\{f_n\}_n, \{x_n\}_n$ sequences of $\Test(M)$ functions as stated in the previous lemma. Let us call $\omega_n:= f_n \de x_n$. By definition, $\omega_n$ is a test $1$-form. We want to prove that $\omega_n \to \omega$ in $W_H^{1,2}(T^*M)$. Using the calculus properties of $\de,\delta$ on test forms (see \cite[Prop. 6.4.3]{gigpas2020notes}, \cite[Prop. 6.4.17]{gigpas2020notes} together with \cite[Prop. 4.2.7]{gigpas2020notes}, and \cite[Eq. 3.5.13]{gigli2018nonsmooth} for the correspondence $\delta \omega_n = -\Div \omega_n^\sharp$), we obtain 
    \begin{align*}
        \omega - \omega_n &= {(f-f_n)} \,{\de x} + f_n \,\de (x-x_n)\\
        \de(\omega - \omega_n) &= \de{(f-f_n)} \wedge {\de x} + \de f_n \wedge \de (x-x_n)\\
        \delta(\omega - \omega_n) &= g({\de(f-f_n)} ,{\de x}) - {(f-f_n)} {\Delta x} + g({\de f_n} ,{\de (x- x_n)}) - f_n{\Delta (x-x_n)}.
    \end{align*}
    The properties of $f_n, x_n$ together with the fact that $f,\de x, \Delta x$ are essentially bounded, ensure that the three right hand sides converge to 0 in $L^2$ as $n\to \infty$.
\end{proof}

We will now denote by $\bm{\Ric}: H^{1,2}_H(TM)\times H^{1,2}_H(TM) \rightarrow \{\text{Radon measures on M}\}$ the measure-valued Ricci curvature tensor defined in \cite[Thm. 6.5.1]{gigpas2020notes}.
\begin{Theorem}\label{thm:RCDtodistributional}
    Let $(M,g)$ be a complete, connected Riemannian manifold with $g\in C^{0,1}_\loc(\cT^0_2 M)$, such that $(M,d,\Vol)$ is an $\mathsf{RCD}(K,\infty)$ space for some $K\in \R$. Then $\Ric \geq K$ in the distributional sense, i.e.
    \begin{align*}
        \Ric(X,X) \geq K |X|^2 \qquad \text{as distributions, for all } X\in C^\infty_c(TM).
    \end{align*}
\end{Theorem}

\begin{proof}
    Take an arbitrary $X\in C^\infty_c(TM)\subseteq H^{1,2}_H(TM)$.  Due to their compact support and their $W^{-1,\infty^-}_\loc$ regularity, Remark \ref{rk:W-1pMultipl} ensures that we are allowed to evaluate all the following distributions on the non-compactly supported canonical volume density $\Vol\in C^{0,1}_{\mathrm{loc}}(\Vol\, M) =W^{1,\infty}_\loc(\Vol\, M)\subseteq W^{1,\infty^-}_\loc(\Vol\, M)$. First of all
    \begin{align*}
        \langle \Delta \abs{X}^2, \Vol\rangle = -\int_M g \left(\nabla \abs{X}^2, \nabla 1 \right) \,\de\Vol=0.
    \end{align*}
    Then by \cite[Thm. 6.5.1]{gigpas2020notes} and Theorem \ref{thm:BochWeitz} the following holds:
    \begin{align*}
        \bm{{\Ric}}(X,X)(M) &=\int_M \vert{\delta X^\flat}\vert^2 + \vert{\de X^\flat}\vert^2 - \abs{\nabla X}^2\,\de\Vol\\
        &=\left\langle \frac{1}{2} {\Delta}\abs{X}^2 +  g(\Delta_H X, X )- \abs{\nabla X}^2 ,\Vol\right\rangle \\
        &=\left\langle\Ric(X,X),\Vol\right\rangle.
    \end{align*}
    We also know from \cite[Thm. 6.5.1]{gigpas2020notes} that, for all $Y\in H^{1,2}_H(TM)$,
    \begin{align*}
        \bm{{\Ric}}(Y,Y) (M)\geq K \int_M \abs{Y}^2 \,\de\Vol.
    \end{align*}
    Given a non negative $\phi \in C^\infty_c(M)$, we can always represent it as a finite sum of squares of $C^{1,1}_c(M)\subseteq W^{1,\infty^-}_\loc(M)$ functions, $\phi = \sum_{i=1}^m \psi_i^2$. 
    A proof of this fact can be found in \cite[Lem. 4]{guan97mongeampere}. Then the following multiplication is justified by Remark \ref{rk:W-1pMultipl}: 
    \begin{align*}
        \left\langle\Ric(X,X),\phi\right\rangle =  \left\langle (\sum_i \psi^2_i)\Ric(X,X),\Vol\right\rangle =  \sum_i \left\langle\Ric (\psi_i X, \psi_i X),\Vol\right\rangle.
    \end{align*}
    Due to \cite[Lem. 3.3]{bregig2024calculus},  we also know that $\psi_i X \in H^{1,2}_H(TM)$. We can thus conclude
    \begin{align*}
        \left\langle\Ric(X,X),\phi\right\rangle &= \sum_i \left\langle\Ric (\psi_i X, \psi_i X),\Vol\right\rangle
        = \sum_i \bm{{\Ric}}(\psi_i X, \psi_i X) (M)\\
        &\geq K \sum_i \int_M \abs{\psi_i X}^2 \,\de\Vol
        = K \int_M \phi \abs{X}^2 \,\de\Vol.
    \end{align*}
\end{proof}

\subsection[Distributional curvature bound implies synthetic curvature bound]{Distributional curvature bound implies synthetic curvature bound - an approximation based proof}
\label{Subsection: Distributional -> synthetic}
In this section we will give an alternative proof of the fact that a lower bound on the distributional Ricci curvature tensor $\Ric \geq K$ of a $C^{0,1}_\loc$ Riemannian manifold $(M,g)$ implies that the metric measure space $(M, d, \Vol)$ is an $\mathsf{RCD}(K,\dim M)$ space. This has been established for metrics of even lower regularity (namely $g \in C^0 \cap W^{1,2}_\loc$) in \cite{mondino2024equivalence}. Our alternative proof, based on the stability of the variable $\mathsf{CD}$-condition established in \cite{ketterer2017variableCD}, is both shorter and has the advantage of providing a slight improvement in the sense that the volume growth condition \cite[Eq.\ (4.1)]{mondino2024equivalence} is not required for $g$ in $C^{0,1}_\loc$. We explicitly construct a pointed measured Gromov-Hausdorff approximating sequence of smooth Riemannian manifolds for $(M,g)$, which constitutes a (potentially) stronger result\footnote{A famous conjecture in the theory of $\mathsf{RCD}$-spaces states that every non-collapsed $\mathsf{RCD}$-space is obtained as a pmGH-limit of smooth Riemannian manifolds with lower Ricci bounds.}.

As we saw in Section \ref{Subsection: C1distribcurvat}, a $C^{0,1}_\loc$ complete connected Riemannian manifold $(M,g)$ can be approximated by a sequence of smooth complete connected Riemannian manifolds $(M,g_i)$ with $g_i\to g$ in $C^0_\loc$. It follows then directly that the pointed metric measure spaces $(M,d_i,\lambda_i^{-1}\Vol_i,p)$ converge in the pointed measured Gromov-Hausdorff sense to $(M,d,\lambda^{-1}\Vol,p)$, where $\lambda_i := \Vol_i(B_1^i(p))$ and $\lambda:=\Vol(B_1(p))$, for any $p\in M$. Let us briefly recall the meaning of these notions, following the definitions given in \cite{gimosa2015pmGHrcd}.
\begin{Definition}
    We call \emph{metric measure space} any triple $(X,d,\meas)$ such that
    \begin{enumerate}
        \item $(X,d)$ is a complete, separable metric space,
        \item $\meas$ is a Radon measure on $X$.
    \end{enumerate}
    We call \emph{pointed metric measure space} any quadruple $(X,d,\meas, x)$ such that $(X,d,\meas)$ is a metric measure space and $x\in \supp \meas$. We say $\Psi$ is an \emph{isomorphism} between the pointed metric measure spaces $(X,d,\meas,x),\, (Y,d',\meas',y)$ whenever $\Psi:\supp \meas\to Y$ is an isometry such that $\Psi_{\#}\meas=\meas'$ and $\Psi(x)=y$.
\end{Definition}
\begin{Definition}
    We say a sequence of pointed metric measure spaces $\{(X_n,d_n,\meas_n,x_n)\}_{n\in\N}$ converges in the \emph{pointed measured Gromov-Hausdorff} (from now on \emph{pmGH}) sense to a metric measure space $(X,d,\meas,x)$ whenever for any $R,\eps >0$ there exists $N(\eps,R)\in \N$ such that for all $n\geq N(\eps,R)$ there exists a Borel map $f^{R,\eps}_n: B_R(x_n)\to X$ such that
    \begin{enumerate}
        \item $f^{R,\eps}_n(x_n) = x$,
        \item $\sup_{y,z\in B_R(x_n)} \abs{d_n (y,z) - d( f^{R,\eps}_n(y), f^{R,\eps}_n(z))} < \eps$,
        \item the $\eps$-neighborhood of $f^{R,\eps}_n(B_R(x_n))$ contains $B_{R-\eps}(x)$,
        \item $(f^{R,\eps}_n)_* (\meas_n\vert_{B_R(x_n)})$ weakly converges to $\meas\vert_{B_R(x)}$, i.e.
        \begin{gather*}
            \int_{B_R(x_n)} \varphi\circ f^{R,\eps}_n \;\de \meas_n \longrightarrow \int_{B_R(x)}\varphi \;\de \meas \\
            \text{for all bounded }\varphi\in C(X) \text{ with bounded support,}
        \end{gather*}
        for a.e. $R>0$.
    \end{enumerate}
\end{Definition}
There exists a closely related notion of convergence of isomorphism classes of pointed metric measure spaces, called \emph{pointed measured Gromov} (in short \emph{pmG}) convergence. Very briefly, it describes the convergence of pointed metric measure spaces $(X_n,d_n,\meas_n,x_n)$ when seen as isometrically embedded subspaces of an encompassing metric space $(Y,\bm{d})$. For a precise definition we point towards \cite{gimosa2015pmGHrcd}. It suffices to know that, when it exists, such a limit is unique. Also we will use the following property (\cite{gimosa2015pmGHrcd}, Proposition 3.30).
\begin{Theorem}\label{thm:GMSpmG}
    If pointed metric measure spaces $(X_n,d_n,\meas_n,x_n)$ converge to $(X,d,\meas,x)$ in the pmGH sense, then their isomorphism classes converge $[X_n,d_n,\meas_n,x_n]\to [X,d,\meas,x]$ in the pmG sense.
\end{Theorem}
We shall require the notion of variable curvature-dimension condition $\mathsf{CD}(\kappa,N)$ introduced by Ketterer in \cite{ketterer2017variableCD}. We will not discuss the definition of $\mathsf{CD}(\kappa,N)$ per se, as it is sufficient for us to know the following facts.
\begin{Theorem}
    Let $(M,g)$ be a complete smooth Riemannian manifold. Let $\kappa:M\to \R$ be a continuous function. Then the metric measure space $(M,d,\lambda\Vol)$, for any $\lambda>0$, satisfies the condition $\mathsf{CD}(\kappa,\dim M)$ if and only if it has Ricci curvature bounded from below by $\kappa$.    
\end{Theorem}
\begin{Theorem}\label{thm:kettererRiemannCD}
    Let $\{(M_i,g_i,o_i)\}_{i\in\N}$ be a sequence of $n$-dimensional pointed smooth Riemannian manifolds that satisfy the condition $\mathsf{CD}(\kappa_i,n)$ for $\kappa_i\in C^0(M_i)$. Let $\meas_i := \Vol_i(B^i_1(o_i))^{-1} \Vol_i$. Assume that
    \begin{equation}\label{eq:CDint}
        {R^{2p}} \int_{B_R^i(o_i)} \left[ (\kappa_i - K)^- \right]^p \,\de\meas_i\xrightarrow{i\to\infty}0 \qquad \forall R>0
   \end{equation}
    for some $K\in \R$ and $p>\frac{n}{2}$.\\
    Then a subsequence of the family of isomorphism classes of the pointed metric measure spaces $\{[M_i, d_i, \meas_i,o_i]\}_i$ converges in the pmG sense to the isomorphism class of a $\mathsf{CD}(K,n)$ space.
\end{Theorem}
These theorems are, respectively, (a special case of) Theorem 2.14 and Theorem 7.1 from \cite{ketterer2021integralCD}.
\begin{Theorem}\label{thm:distribtosynth}
    Let $(M,g)$ be a complete, connected smooth manifold endowed with a $C^{0,1}_\loc$ Riemannian metric $g$. Assume there exists $K\in \R$ such that $\Ric\geq K$ in the distributional sense. Then the metric measure space $(M, d, \Vol)$ is an $\mathsf{RCD}(K, \dim M)$ space.
\end{Theorem}
\begin{proof}
    Let $n=\dim M$. Consider as before $g_i= g_{\varepsilon_i}$ as in Remark \ref{Remark: Approximatingmetrics} for some $\eps_i\searrow 0$. It follows from Theorem \ref{thm:C1hopfrinow}, together with \eqref{eq:balls_g_geps} that  $(M,g_i)$ is a complete Riemannian manifold for every $i$. Observe that for all $i\in \N$ one can define the continuous function $\kappa_i:M\to \R $ assigning to every $p\in M$ the lowest eigenvalue of $\Ric_i (p)$ with respect to $g_i(p)$. Clearly, it holds that $\Ric_i \geq \kappa_i g_i$. 
    
    Fix any $x_0\in M$. In order to apply Theorem \ref{thm:kettererRiemannCD} we need to verify that, for some $p>\frac{n}{2}$,
    \begin{align}\label{eq:ketterer_to_verify}
        \frac{R^{2p}}{\Vol_i (B_1^i(x_0) )} \int_{B_R^i(x_0)} \left[ (\kappa_i - K)^- \right]^p \,\de\Vol_i\xrightarrow{i\to\infty}0 \qquad \forall R>0.
    \end{align}
To see this, note first that, by \eqref{eq:balls_g_geps},
for any fixed $R>0$ and all $i\in \N$ we have $B_R^i(x_0) \subseteq \overline{B_{2R}(x_0))} \Subset M$, and since $g_i\to g$ locally uniformly, for large $i$ we additionally have
$\Vol_i (B_1^i(x_0) )$ $\ge$ $\Vol_i(B_{1/2}(x_0)) \ge \frac{1}{2}\Vol(B_{1/2}(x_0))$. Thus to establish
\eqref{eq:ketterer_to_verify}, it remains to show that $(\kappa_i - K)^- \to 0$ in $L^p_{\mathrm{loc}}(M)$. Fix some $L\Subset M$, and without loss of generality suppose that $L$ is small enough to admit a smooth local frame $X_1,\dots,X_n$ in a neighborhood $V$ of $L$. Also, fix $\bar x\in L$. Then according to the Rayleigh quotient formula, for any $i$ there exists some $g_i$-unit vector $\bar X^{(i)}\in T_{\bar x}M$ with $(\kappa_i(\bar x)-K) = \Ric_{g_i}(\bar X^{(i)},\bar X^{(i)})-K$. 
Let $\bar X^{(i)}=\sum_{k=1}^n \alpha_k^{(i)} X_k(\bar x)$ ($\alpha_1^{(i)},\dots,\alpha_n^{(i)}\in \R$) and set
\begin{equation}\label{eq:Xidef}
    X^{(i)} := \sum_{k=1}^n \alpha_k^{(i)} X_k
\end{equation}
on $V$. Since $\bar X^{(i)}$ is a $g_i$-unit
vector it follows from \eqref{geps3} and the local equivalence of $\|\,\|_g$ with the Euclidean norm w.r.t.\ the frame $X_k$ that
there exists some
$C_L>0$ such that
\begin{equation}\label{eq:alpha_bounds}
|\alpha_k^{(i)}| \le C_L \qquad \forall i\in \N,\ k=1,\dots,n,
\end{equation}
independently of the choice of $\bar x\in L$ and $\bar X^{(i)}$ as above.
On $L$ we have
\begin{equation}\label{eq:Ric_i_decomposition}
\begin{split}
\Ric_i(X^{(i)},X^{(i)}) -& Kg_i(X^{(i)},X^{(i)}) = 
[\Ric_i - (\Ric\star_M \rho_i)](X^{(i)},X^{(i)}) \\
&+[(\Ric\star_M \rho_i)(X^{(i)},X^{(i)}) - \Ric(X^{(i)},X^{(i)})\star_M\rho_i]\\
&+[\Ric(X^{(i)},X^{(i)}) - Kg(X^{(i)},X^{(i)})]\star_M\rho_i \\
&+ [Kg(X^{(i)},X^{(i)})\star_M\rho_i - Kg_i(X^{(i)},X^{(i)})] \\
&=: A_i(X^{(i)},X^{(i)}) + B_i(X^{(i)},X^{(i)}) + C_i(X^{(i)},X^{(i)}) + D_i(X^{(i)},X^{(i)}).
\end{split}
\end{equation}
Since $\Ric(X,X)\geq K g(X,X)$ in distributions and $\rho\ge 0$, $C(X^{(i)},X^{(i)}) \ge 0$. Therefore,
\[
[\Ric_i(X^{(i)},X^{(i)}) - Kg_i(X^{(i)},X^{(i)})]^- \le |A_i(X^{(i)},X^{(i)})| + |B_i(X^{(i)},X^{(i)})| + |D_i(X^{(i)},X^{(i)})| 
\]
pointwise on $L$. Set $\hat A_i := \max\{|A_i(X_k,X_l)| \colon k,l = 1,\dots,n \}$, and analogously define $\hat B_i$ and
$\hat D_i$.
Inserting \eqref{eq:Xidef} into \eqref{eq:Ric_i_decomposition} and taking into account \eqref{eq:alpha_bounds}, it follows that there exists a constant $\tilde C_L$ such that
\[
[\Ric_i(X^{(i)},X^{(i)}) - Kg_i(X^{(i)},X^{(i)})]^- \le
\tilde C_L(\hat A_i + \hat B_i + \hat D_i),
\]
pointwise on $L$. 
Since all of the above arguments apply irrespective of the choice of $\bar x$ and $\bar X^{(i)}$ we conclude that $[\kappa_i-K]^-\le
\tilde C_L(\hat A_i + \hat B_i + \hat D_i)$, pointwise on $L$.
Proposition
\ref{prop:LpRicConvergence} shows that both $\hat A_i$ and $\hat B_i$ converge to $0$ in $L^p(L)$. Moreover, $\hat D_i\to 0$ uniformly on $L$. Altogether, this establishes \eqref{eq:ketterer_to_verify}.
 
From this, by applying Theorem \ref{thm:kettererRiemannCD}, we obtain that $\{[M, d_i, \meas_i, x_0]\}_{i\in\N}$ admits a subsequence that converges in the pmG-sense to the isomorphism class of a $\mathsf{CD}(K, \dim M)$ space. At the same time, due to the convergence $g_i\to g$ in $C^0_\loc$, we can easily see that $(M, d_i , \meas_i, x_0) \to (M,d,\meas,x_0)$, where $\meas_{(i)}:= \Vol_{(i)}(B_1(x_0))^{-1} \Vol$, in the pmGH-sense. By Theorem \ref{thm:GMSpmG}, also their isomorphism classes converge in the pmG-sense. Due to uniqueness of the limits, and the fact that the $\mathsf{CD}(K,N)$ condition is invariant under isomorphisms of metric measure spaces, we have that $(M,d,\meas,x_0)$ is a $\mathsf{CD}(K,N)$-space. Also the Riemannian structure of $(M,g)$ naturally ensures the infinitesimal Hilbertianity of $(X,d,\meas)$ (cf.\ also the discussion on Sobolev spaces from the previous section), making it an $\mathsf{RCD}(K,N)$-space. Since $\Vol$ is simply a rescaling of $\meas$ and the $\mathsf{RCD}(K,N)$-condition is invariant under this operation, the claim follows.
\end{proof}

\subsection{Consequences of the equivalence}
\label{Subsection: Consequencesofequivalence}
The main consequence of the equivalence of the distributional Ricci curvature bound and the $\mathsf{RCD}$ condition is that it allows one to apply to Riemannian manifolds with a low-regularity metric the many theorems that have been proven in the $\mathsf{RCD}$ setting, mostly generalizations of well known theorems from the smooth theory. In this section we discuss some of the more immediate or useful of these results.\\
Before stating the first results, we need to clarify that when we talk about geodesics there are two related but non-coinciding notions. In Riemannian geometry geodesics are curves $\gamma$ which solve the geodesic equation
\begin{equation*}
    \Ddot{\gamma}_k(t) + \Gamma^{ij}_k \Dot{\gamma}_i(t)\Dot{\gamma}_j(t)=0,
\end{equation*}
while in metric geometry we often call geodesics the curves such that minimize the distance, i.e. $L(\gamma)= d(\gamma(0),\gamma(1))$, where $L$ is the curve length functional. If the metric $g$ is at least $C^{1,1}$, then it is well-known that solutions of the geodesic equation locally minimize the length functional (see \cite[Thm.\ 6]{minguzzi2015convex}). However, in \cite{hartman1951problems} (see also \cite{samstein2018geodlowreg}) one may find examples of $C^{1,\alpha}$-metrics (with $\alpha \in (0,1)$) for which there exist solutions to the geodesic equation that do not minimize the length functional anywhere even locally. 

We will refer to geodesics in the metric sense as \emph{minimizing geodesics}. Due to \cite[Thm. 1.4]{lytyam06geod}, those are $C^{1,1}$ curves whenever $g\in C^{0,1}_\loc(\cT^0_2 M)$.

We already know from the metric Hopf-Rinow Theorem that on a complete connected Riemannian manifold with $C^0$ metric any two points admit a minimizing geodesic between them. Adding the $\mathsf{RCD}(K,N)$-condition gives us additional properties for minimizing geodesics.
\begin{Proposition}
    Let $(M,g)$ be a complete, connected Riemannian manifold with $g\in C^{0,1}_\loc(\cT^0_2 M)$ and suppose that $\Ric\ge K$ in the distributional sense for some $K\in \R$. Then the following properties hold:
    \begin{enumerate}
        \item Given any $p\in M$, for $\Vol$-a.e. $q\in M$ there exists a unique minimizing geodesic with endpoints $p$ and $q$.
        \item $(M,d)$ is a non-branching metric space, i.e.\ if there exist two minimizing geodesics $\gamma,\Tilde{\gamma}:[0,1]\to M$ and $t\in(0,1)$ such that $\gamma\vert_{[0,t]} = \Tilde{\gamma}\vert_{[0,t]}$, then $\gamma = \Tilde{\gamma}$.
    \end{enumerate}
\end{Proposition}
Both statements have been proven on $\mathsf{RCD}(K,N)$-spaces or on spaces satisfying the $\mathsf{RCD}^*(K,N)$-condition, which has been proven to be equivalent in \cite{cavmil2021global} (see \cite{girast2016optimalexp} and \cite{deng2020branching}). Though technically the second result does not admit $N=1$, it is sufficient to recall that $\mathsf{RCD}(K,1)$ implies $\mathsf{RCD}(K,2)$.

A well-known result in the theory of synthetic curvature bounds is the extension of the Cheeger-Gromoll splitting theorem to $\mathsf{RCD}(0,N)$, proven in \cite{gigli2013splitting} (see \cite{gigli2014splitoverview} for a more user-friendly overview).
\begin{Theorem}
    Let $(X,d,\meas)$ be an $\mathsf{RCD}(0,N)$ space with $1\leq N <\infty$. Assume $\supp \meas$ contains a line, i.e. a curve $c:\R \to X$ such that
    \begin{equation*}
        d(c(t),c(s))= \abs{t-s}\qquad \text{for all }t,s\in\R.
    \end{equation*}
    Then there exists a metric measure space $(X',d',\meas')$ such that $X$ is isomorphic to $X'\times \R$ endowed with the product measure $\meas'\times\mathcal{L}^1$ and the product distance
    \begin{equation*}
        d_{X'\times \R}((x',t)(y',s)):= \sqrt{d'(x',y')^2 + \abs{t-s}^2 }.
    \end{equation*}
    Furthermore if $N\geq 2$, then $(X',d',\meas')$ is an $\mathsf{RCD}(0,N-1)$ space. For $N\in [1,2)$, $X'$ is a singleton.
\end{Theorem}
This result can clearly be applied to a much wider class of metric measure spaces than the low-regularity Riemannian manifolds we took into consideration. On the other hand our version of the splitting theorem, Theorem \ref{thm:SplittingTheorem}, ensures that the splitting map has higher regularity than just Lipschitz.

Other theorems proven on $\mathsf{RCD}$ spaces which could prove useful tools for low-regularity Riemannian geometry include:
\begin{enumerate}
    \item a generalized Bishop-Gromov inequality for volumes of balls and surface measures of spheres and a generalized Brunn-Minkowski inequality proven in \cite{Stu:06a, Stu:06b};
    \item a generalized Bonnet-Meyers theorem, also proven in \cite{Stu:06a, Stu:06b} (see also \cite[App.]{graf2020singularity} for a distributional version);
    \item a generalized L\'evy-Gromov isoperimetric inequality, proven in \cite{cavmon2017sharp};
    \item a generalized Poincar\'e inequality, proven in \cite{rajala2012local}.
\end{enumerate}

\section{Conclusion \& outlook}
\label{Section: Conclusionandoutlook}

In this work, we have established the Cheeger-Gromoll splitting theorem for $C^1$-Riemannian metrics (see Theorem \ref{thm:SplittingTheorem}), where we were able to obtain a higher regularity of the splitting isometry than the Lipschitz regularity one automatically obtains from the $\mathsf{RCD}$-splitting theorem \cite{gigli2013splitting, gigli2014splitoverview}. We clarified the case of metrics of higher regularity in Theorem \ref{thm: higherregularitysplitting}. An essential tool to obtain our splitting results was a generalized non-smooth Bochner-Weitzenböck formula (see Theorem \ref{thm:BochWeitz}). We also obtained the flatness criterion for semi-Riemannian metrics of regularity $C^1$, giving the three usual characterizations via vanishing of the Riemann curvature tensor, the existence of local parallel orthonormal frames, and the existence of local flat coordinates (see Theorem \ref{Theorem: FlatnesscriterionC1} and Corollary \ref{Corollary: FlatnessforCkmetrics}). Finally, we discussed the relationship between various notions of Sobolev spaces (see Subsection \ref{Subsection: CompatibilitybetweennotionsofSobolev}), and gave alterative proofs for locally Lipschitz metrics (see Theorems \ref{thm:RCDtodistributional} and \ref{thm:distribtosynth}) of the equivalence between synthetic and distributional Ricci curvature lower bounds established in \cite{mondino2024equivalence} for metrics of regularity $C^0 \cap W^{1,2}_\loc$.

Various open problems remain regarding curvature in low regularity. For example, it would be of great interest to generalize the results of \cite{mondino2024equivalence} in several directions, e.g.\ the low-regularity Finslerian setting (dropping the Riemannian condition), even lower regularity of the metric (perhaps even for Geroch-Traschen metrics; see \cite{geroch1987strings, steinbauer2009geroch}), or to the Lorentzian signature (for this, one can expect that a differential calculus in the synthetic Lorentzian setting as developed in \cite{Octet24+} will be essential). Dropping the regularity of the metric in the rigidity theorems (the splitting theorem and the flatness criterion) even further is also worth investigating. Finally, a Lorentzian low regularity version of the Cheeger-Gromoll rigidity theorem (see \cite{eschenburg1988splitting, galloway1989lorentzian, newman1990proof} for the smooth case) is of importance to mathematical General Relativity, and is currently work in progress \cite{Quintetellipticsplitting24}.

\section*{Data availability statement}
There is no data associated with this manuscript.
\section*{Conflict of interest statement}
The authors confirm that no conflict of interest exists for this work.

\section*{Acknowledgments}
This research was funded in part by the Austrian Science Fund (FWF) [Grants DOI \linebreak \href{https://doi.org/10.55776/PAT1996423}{10.55776/PAT1996423}, \href{https://doi.org/10.55776/P33594}{10.55776/P33594}, and \href{https://doi.org/10.55776/EFP6}{10.55776/EFP6}]. For open access purposes, the authors have applied a CC BY public copyright license to any author accepted manuscript version arising from this submission. Argam Ohanyan is also supported by the ÖAW-DOC scholarship of the Austrian Academy of Sciences. Alessio Vardabasso acknowledges the support of the Vienna School of Mathematics.

% \bibliography{Bibliography} 
% \bibliographystyle{acm}
%\bibliographystyle{alpha}
\end{document}